\input ./style/arxiv-general.cfg
\documentclass[MSNbibl,number,citesort]{arxbj}
\makeatletter
   \@ifpackageloaded{graphicx}{}{\usepackage{graphicx}}
\makeatother
\usepackage{upgreek}
\usepackage{mathrsfs}
\usepackage[all]{xy}

\volume{22}
\issue{3}
\pubyear{2016}
\firstpage{1535}
\lastpage{1571}
\doi{10.3150/15-BEJ703}
\docsubty{FLA}

\makeatletter
\def\pi{\uppi}
\newcommand{\eqref}[1]{(\ref{#1})}

\newtheorem{theorem}{Theorem}[section]
\newtheorem{lemma}{Lemma}[section]
\newproclaim{definition}{Definition}[section]

\newremark{remark}{Remark}
\newremark{remarks}{Remarks}
\newcommand{\Nat}{\mathbb{N}}
\newcommand{\real}{\mathbb{R}}
\newcommand{\Cb}{C_b}

\newcommand{\Bcal}{\mathcal{B}}

\newcommand{\Dcal}{\mathcal{D}}
\newcommand{\Ecal}{\mathcal{E}}
\newcommand{\Fcal}{\mathcal{F}}
\newcommand{\Kcal}{\mathcal{K}}
\newcommand{\Pcal}{\mathscr{P}}
\newcommand{\Qcal}{\mathcal{Q}}
\newcommand{\Gcal}{\mathscr{G}}
\newcommand{\TPlan}{\mathcal{T}}

\newcommand{\BB}{\mathscr{B}}
\newcommand{\gauge}{g}
\newcommand{\m}{k}

\newcommand{\Prob}{\mathbb{P}}
\newcommand{\bd}{\operatorname{bd}}
\newcommand{\Diam}{\operatorname{diam}}
\newcommand{\supp}{\operatorname{spt}}
\newcommand{\DP}{\mathscr{D}}
\newcommand{\Cclass}{\mathscr{C}}
\newcommand{\law}{\operatorname{law}}

\newcommand{\W}{W}
\newcommand{\Wr}{W_r}

\newcommand{\Wcalr}{\W_r}

\newcommand{\coupling}{\kappa}
\newcommand{\Gz}{G_0}
\newcommand{\hatQ}{\hat{Q}}

\newcommand{\pGone}{p_{Y_{[1]}|G}}

\newcommand{\PGz}{P_{Y_{[n]}|G_0}}
\newcommand{\PG}{P_{Y_{[n]}|G}}

\newcommand{\hatf}{\hat{f}}

\newcommand{\Bky}{B_K}

\newcommand{\YYn}{Y_{[n]}}
\newcommand{\YYone}{Y_{[1]}}

\newcommand{\vecq}{\mathbf{q}}

\newcommand{\epsmn}{\varepsilon_{mn}}
\newcommand{\delmn}{\delta_{mn}}

\newcommand{\nn}{\tilde{n}}

\makeatother

\begin{document}
\begin{frontmatter}

\title{Borrowing strengh in hierarchical Bayes: Posterior
concentration of the Dirichlet base measure}
\runtitle{Posterior concentration of the Dirichlet base measure}

\begin{aug}
\author[A]{\inits{X.}\fnms{XuanLong}~\snm{Nguyen}\corref{}\ead[label=e1]{xuanlong@umich.edu}}
\address[A]{Department of Statistics,
University of Michigan,
456 West Hall,
Ann Arbor, MI 48109-1107,
USA.\\ \printead{e1}}
\end{aug}

%
\received{\smonth{11} \syear{2013}}
%
\revised{\smonth{10} \syear{2014}}

%
\begin{abstract}
This paper studies posterior concentration behavior of the base
probability measure of a Dirichlet measure,
given observations associated with the sampled Dirichlet processes,
as the number of observations tends to infinity.
The base measure itself is endowed with another Dirichlet prior,
a construction known as the hierarchical Dirichlet
processes (Teh \textit{et al.}
[\textit{J. Amer. Statist. Assoc.} \textbf{101} (2006) 1566--1581]).
Convergence rates are established in transportation distances (i.e.,
Wasserstein metrics) under various
conditions on the geometry of the support of the true base measure.
As a consequence of the theory, we demonstrate
the benefit of ``borrowing strength'' in the inference of multiple
groups of data -- a powerful insight often invoked to motivate
hierarchical modeling.
In certain settings, the gain in efficiency due to the latent hierarchy
can be dramatic, improving
from a standard nonparametric rate to a parametric rate of
convergence. Tools developed include transportation distances for
nonparametric Bayesian hierarchies of random measures,
the existence of tests for Dirichlet measures,
and geometric properties of the support of Dirichlet measures.
\end{abstract}

%
\end{frontmatter}

\section{Introduction}\label{sec1}

Ferguson's Dirichlet process is a fundamental building block in nonparametric
Bayesian statistics \cite{Ferguson,Blackwell-MacQueen,Sethuraman}.
Recent advances in modeling and computation have seen Dirichlet processes
routinely built into hierarchical probabilistic structures in
innovative ways \cite{Hjort-etal-10}.
A particularly useful and interesting structure that is also the
focus of this paper, is
the hierarchical Dirichlet processes \cite{Teh-etal-06,Teh-Jordan-10} --
a construction in which
the base probability measure of the Dirichlet becomes an object of
inference, which is endowed with yet another Dirichlet prior.
The hierarchical Dirichlet processes have been successfully applied to
the problem of clustering for grouped data in
a vast array of domains.\footnote{The Google scholar page shows
more than 1400 citations of \cite{Teh-etal-06}.}


This paper investigates the asymptotic behavior of measure-valued
latent variables that arise in the hierarchical Dirichlet processes.
The basic question that we address
is the convergence of an estimate of the base probability measure
(hereafter ``base measure'') of a Dirichlet measure, given observations
associated with the Dirichlet processes sampled by
the Dirichlet.
Let $\Theta$ be a complete separable metric space equipped with the Borel
sigma algebra, $\Pcal(\Theta)$ the space of probability measures on
$\Theta$, and
let $G \in\Pcal(\Theta)$ and $\alpha> 0$.
Recall from \cite{Ferguson}
that a Dirichlet process $Q$ is a random measure taking value in $\Pcal
(\Theta)$
and distributed by a Dirichlet measure
$\DP_{\alpha G}$, if for any measurable partition $(B_1,\ldots, B_k)$ of
$\Theta$ for some $k \in\Nat$, $(Q(B_1),\ldots,Q(B_k))$ is
a random vector
distributed according to the $k$-dimensional Dirichlet distribution
with parameters $(\alpha G(B_1),\ldots, \alpha G(B_k))$.

\textit{Questions}.
Let $Q_1,\ldots, Q_m$ be an i.i.d. $m$-sample from a
Dirichlet measure $\DP_{\alpha G}$, where $\alpha> 0$ is
given and the base measure $G=G_0$ is unknown.
By a basic property of Dirichlet processes, $Q_1,\ldots, Q_m$ are
almost surely discrete probability measures on $\Theta$.
They will \emph{not} be observed directly.
Instead, for each $i=1,\ldots,m$, we shall be given an i.i.d. $n$-sample
$\YYn^i = (Y_{i1},\ldots,Y_{in})$
from a mixture distribution in which $Q_i$ serves as a mixing
measure. This mixture distribution admits the density function
$p_{Q_i}(x):= Q_i*f(x) := \int f(x|\theta) Q_i(\mathrm{d}\theta)$,
where $f(\cdot|\cdot)$ is a known kernel density function defined with
respect to a
dominating measure on $\Theta$.

To estimate $G_0$ by taking a Bayesian approach, the base measure $G$
is endowed
with a prior on the space of measures $\Pcal(\Theta)$, yielding
a hierarchical model specification as follows:
%
%
\begin{eqnarray}
\label{Eqn-HM-1}&& G \sim\Pi_G,\qquad  Q_1,\ldots, Q_m
| G \stackrel{\mathrm{i.i.d.}} {\sim} \DP_{\alpha G},
\\
\label{Eqn-HM-2}&& Y_{i1},\ldots,Y_{in} | Q_i
\stackrel{\mathrm{i.i.d.}} {\sim} Q_i*f \qquad\mbox{for } i=1,\ldots,m.
\end{eqnarray}
For the choice of prior $\Pi_G := \DP_{\gamma H}$,
where $\gamma> 0$ and $H\in\Pcal(\Theta)$ is nonatomic and known,
this construction
is called the hierarchical Dirichlet processes model \cite{Teh-etal-06}.
Fast computational
methods have been developed to collect samples from the posterior
distributions of interest, such as those for the latent $G$ and $Q_i$,
given the $m\times n$ data set $\YYn^{[m]}:= (\YYn^1,\ldots,\YYn^m)$.
The first question considered in this paper is the following:
\begin{longlist}[(I)]
\item[(I)] How fast does the posterior distribution of the base
measure $G$
concentrate toward the true $G_0$, as $m$ and $n$ tend to infinity?
\end{longlist}

An appealing aspect well appreciated by
(Bayesian) modelers and practioners of hierarchical modeling is
the notion of ``borrowing strength''. Latent
variables shared higher up in a conditional independence
probabilistic hierarchy
provide an infrastructure through which one may improve the inference
of a parameter
of interest by borrowing from information on other related data and
parameters that are also part of the model.
For the hierarchical Dirichlet processes, the ``borrowing''
has a particularly concrete meaning: according to
the model, the Dirichlet processes $Q_i$ for all $i=1,\ldots,m$ share
the same set of supporting atoms as that of the base measure $G$. It is
intuitive that the inference of the supporting atoms of, say, $Q_1$ for
group 1,
should benefit from the information given by other groups of data associated
with $Q_2, Q_3$ and so on.
To quantify this intuition, we ask the following:
\begin{longlist}[(II)]
\item[(II)] What is the posterior concentration behavior
of a mixture distribution, denoted by $Q*f$, as $Q$ is attached
to the Bayesian hierarchy in the same way as the $Q_i$,
in comparison to a ``stand-alone'' mixture model $Q*f$,
where $Q$ is endowed with an independent prior distribution?
\end{longlist}
By resolving question (I), we can demonstrate
situations in which the Bayesian hierarchy has the effect of translating
the posterior concentration behavior of base measure $G$
to improved posterior concentration of each individual group of data
in the setting of question (II).
Both questions will be addressed
using the tools that we develop with transportation distances \cite
{Villani-08}.

\textit{Related work}.
The only work known to us about the inference of the Dirichlet
base measure is by \cite{Korwar-Hollander}, who show that it is
possible to obtain a
consistent estimate (in some sense) of a base measure $G_0$, given an
i.i.d. $n$-sample from $m=1$ Dirichlet process $Q_1$ distributed by
$\DP_{\alpha G_0}$. This curious result is due
to two crucial assumptions made in their work:
the true base measure $G_0$ is nonatomic, \emph{and} $Q_1$ is
observed directly. Due to the fact that two Dirichlet measures
with different nonatomic base measures are orthogonal,
the estimation of nonatomic base measures becomes somewhat
simple if the sampled Dirichlet processes $Q_i$ are observed
directly. Changing at least one of the two assumptions makes
the question considerably more difficult, which leads to
different answers and requires new proof techniques.
In this paper, we study the case $G_0$ is an atomic measure
with either finite or infinite support, and the $Q_i$ are
\emph{not} observed directly. To get a sense of the challenge, consider
the simplest case, that
the base measure $G_0$ has a finite number of support
points, say $G_0 = \sum_{i=1}^{k}\beta_i \delta_{\theta_i}$, where
$\theta_1,
\ldots,\theta_k$ are \emph{known}. Having a single observation
$Q_1$ distributed by $\DP_{\alpha G_0}$ is equivalent to being given a
single draw from
a~\mbox{$k$-dim} Dirichlet distribution with parameter
$(\alpha\beta_1,\ldots,\alpha\beta_k)$.
It is clearly impossible to obtain a consistent estimate
of $G_0$ by setting $m=1$ (or finite), and $n\rightarrow\infty$.
In addition, the assumption that $Q_1,\ldots,Q_m$ are \emph{not}
observed directly makes the analysis considerably more delicate, due to
the fact that we would no longer have access to a simple point estimate of
the Dirichlet base
measure, as allowed in \cite{Korwar-Hollander}. We leave
open the setting where $G_0$ is nonatomic \emph{and} the $Q_i$ are not
observed directly. For this setting, the choice of Dirichlet prior in
the hierarchical
Dirichlet processes may not be appropriate, due to the discreteness
of Dirichlet processes. On the other hand, there is no known
practical estimation method available for this setting at the moment.

The convergence theory of posterior distributions has received much
attention in the past decade. Recent references include \cite
{Barron-Shervish-Wasserman-99,Ghosal-Ghosh-vanderVaart-00,Shen-Wasserman-01,Ghosh-Ramamoorthi-02,Walker-04,Walker-Lijoi-Prunster-07}.
See \cite{Ghosal-10} for a concise overview. This theory when applied
to density estimation problem
has become quite mature -- the dominant theme is
a Hellinger theory of density estimation for observed data.
On the other hand, asymptotic behaviors of latent variable models
remain poorly understood. When the inference
of a latent variable is of primary concern, the Hellinger theory alone
is not adequate; moreover, the underlying geometry of the variables of
interest has to be taken into account.
There are some examples of such theory that
have been developed recently, for example, for models of random
functions \cite{vanderVaart-vanZanten-08a,Gine-Nickl}, mixture
models \cite{Rousseau-Mengersen-11,Nguyen-11,Gassiat-vanHandel},
models of random polytopes \cite{Nguyen-admixture}.
In a prior work, the author demonstrated
the usefulness of Wasserstein distances in analyzing the convergence
of latent mixing measures in mixture models \cite{Nguyen-11}.
This viewpoint will be deepened and generalized in this work
to a canonical class of hierarchical models equipped with
optimal transport distances for hierarchies
for random measures.

Latent hierarchies of random variables
have long been a versatile and highly effective
modeling tool for Bayesian modelers (see, e.g., \cite{Berger}).
They can also be viewed as a device for
frequentist concepts of shrinkage and random effects
(see, e.g., Chapter~5 of \cite{Lehmann-Casella}).
Due to their wide usages, it is of interest to characterize
the roles of latent hierarchies and their effects on posterior
inference in
a rigorous manner. Examples of hierarchical \emph{and} parametric
models that have been explored recently include the work
by \cite{Gassiat-Rousseau}, who studied hidden Markov models,
and by the author \cite{Nguyen-admixture}, who studied the finite
admixtures for categorical data.
Theoretical work addressing hierarchical \emph{and} nonparametric
models, remains scarce in the literature.

\textit{Overview of results}.
The contributions of this paper
include: (1) an analysis of convergence for the estimation of
the base measure (mean measure) of a Dirichlet measure,
as well as the convergence behavior of the induced marginal
density of observed data;
(2) a theoretical analysis of the effect of
``borrowing of strength'' in the latent nonparametric hierarchy
of variables;
and (3) as part of the proofs of these two results
we develop new tools that help to explain
the geometry of the support of Dirichlet measures, and the geometry
of test sets that discriminate among different Dirichlet measures.
As mentioned
earlier, our geometric theory is equipped with Wasserstein distances,
and a new class of transportation distances that we will
introduce.

Recall that for $r\geq1$, the $L_r$ Wasserstein distance between
two probability measures $G,G'\in\Pcal(\Theta)$ is given as
%
%
\begin{equation}
\label{Eqn-W-def} \Wr \bigl(G,G' \bigr) = \inf_{\coupling\in\TPlan(G,G')}
\biggl[\int\bigl\|\theta- \theta'\bigr\|^r \,\mathrm{d}\coupling
\bigl(\theta, \theta' \bigr) \biggr]^{1/r}.
\end{equation}
Here, $\TPlan(G,G')$ is the space of all joint distributions
on $\Theta\times\Theta$ whose marginal distributions are $G$ and $G'$.
Such a joint distribution $\coupling$ is also called a coupling
between $G$ and $G'$ \cite{Villani-08}.

There are three main theorems summarized in Section~\ref{Sec-descript}.
Our first main result (Theorem~\ref{Thm-main-0}) establishes the
posterior concentration behavior for the marginal density
$\PG$ of a generic $n$-vector $\YYn= (Y_1,\ldots,Y_n)$, which is obtained
by integrating out the latent variable $Q$ (see the formulae
of the density in equation \eqref{Eqn-pG}).
Suppose that the $m\times n$ data set $\YYn^{[m]} :=
(\YYn^1,\ldots,\YYn^m)$ are generated by the model specified
by equations \eqref{Eqn-HM-1} and \eqref{Eqn-HM-2},
according to $G= G_0$ for some unknown $G_0 \in\Pcal(\Theta)$,
where $\Theta$ is taken to be a bounded subset of $\real^d$.
For each fixed $n$, as $m\rightarrow\infty$, there is a vanishing
sequence $\epsmn= C[(n^{3d})\log(mn)/m]^{1/(2d+2)}$ such
that the posterior probability
%
%
\begin{equation}
\label{Eqn-concent-hellinger} \Pi_G \bigl(h(p_{Y_{[n]}|G_0},
{p_{Y_{[n]}|G}}) \leq\epsmn| \YYn^{[m]} \bigr)
\longrightarrow1
\end{equation}
in $\PGz^m$-probability. Here, $\PGz^m$ denotes the true probability
measure that generates the data set, $C$ is a constant independent of
$m$ and $n$, and $h$ denotes the Hellinger distance. Moreover,
equation \eqref{Eqn-concent-hellinger} continues to hold if we allow
$n:= n(m)$ to increase
(e.g., to infinity) as well. This concentration rate holds under minimum
assumptions on the kernel density $f$ of the mixture distributions.
In fact, improved rates can be achieved when more is assumed about
either $f$ or $G_0$. For instance, if $f$ is a standard Gaussian
kernel, then
$\epsmn\asymp
[n^{2d}(\log n)(\log m)^{2d+1}/m]^{1/2}$, which is optimal in terms
of $m$ (up to a logarithmic quantity). This is quite noteworthy since
$G_0$ may have infinite support. On the other hand, if we consider
a hierarchical parametric setting, that is,
$G_0$ has finite and known number of support points, while $f$ is
an arbitrary kernel satisfying some mild conditions, then
we obtain parametric rate $\epsmn\asymp[\log(mn)/m]^{1/2}$.

Our second main result (Theorem~\ref{Thm-main-1} in Section~\ref
{Sec-descript})
turns to the posterior concentration behavior of base measure $G$.
In numerous applications of the hierarchical Dirichlet processes
to biomedical and machine learning problems \cite{Teh-etal-06},
the practitioners are usually not interested in the
marginal densities of the observed groups of data per se, but rather
the inference
of the latent variables $Q_i$ and $G$, as they represent
specific information about the underlying heterogeneity in
data population. In admixed modeling of population genetics, for instance,
$G$ encodes the population structures responsible for
diverse genotypic patterns. In the topic modeling of
documents and images, $G$ may represent topics and objects,
respectively, of the observed texts and visual scenes.

As we shall see, the posterior concentration of the marginal
densities of the data can be shown to entail the concentration
of the base measure $G$, provided (again) that the data are generated
according to some true base measure $G = G_0$.
In this asymptotic result, we work in the regime where $m \rightarrow
\infty$,
while $n := n(m)$ is also taken to increase at an arbitrary rate
relative to $m$. We will show that
%
%
\begin{equation}
\label{Eqn-concentration} \Pi_G \bigl(W_1(G,G_0) \leq
\epsmn+ \Delta_n | \YYn^{[m]} \bigr) \longrightarrow1
\end{equation}
in $\PGz^m$-probability, where $\epsmn$ is the posterior concentration
rate of the marginal densities as established in the previous theorem
(cf. equation \eqref{Eqn-concent-hellinger}).
Quantity $\Delta_n \rightarrow0$ as $n\rightarrow\infty$,
and can be defined as a function of the \emph{demixing} rate $\delta_n$
of a deconvolution problem (cf. \cite
{Carroll-Hall-88,Zhang-90,Fan-91,Nguyen-11}).
[To be clear, $\delta_n$ is the rate of convergence -- in $W_2$ in our
case -- for estimating a mixing measure $Q$
given an i.i.d. $n$-sample of a mixture density $Q*f$.]
The nature of the dependence of $\Delta_n$ on $\delta_n$ is
interesting, as it hinges on
the geometry of the support of the true base measure $G_0$. We can establish
a sequence of gradually deteriorating rates as
the support of $G_0$ becomes less sparse:
\begin{longlist}[(iii)]
\item[(i)] If $G_0$ has a finite and known number of support points on
a bounded
subset of $\real^d$, then
$\Delta_n \asymp\delta_n^{{\alpha^*}}$. In fact, we obtain the overall
parametric rate of convergence under some conditions that
$\epsmn+ \Delta_n \asymp[\log(mn)/m]^{1/2} + [(\log
n)^{1/2}/n^{1/4}]^{{\alpha^*}}$,
where constant ${\alpha^*}= \inf_{\theta\in\supp G_0} \alpha G_0(\{
\theta\})$.
\item[(ii)] If
$G_0$ has a finite and unknown number of support points on a bounded
subset of $\real^d$, then
$\Delta_n \asymp\delta_n^{{\alpha^*}/({\alpha^*}+1)}$.
\item[(iii)] If $G_0$ has an infinite number of
geometrically sparse support points on a bounded subset of $\real^d$,
then $\Delta_n \asymp\exp-[\log(1/\delta_n)]^{1/(1\vee\gamma
_0+\gamma_1)}$
for \emph{supersparse} measures,
or $\Delta_n \asymp[\log(1/\delta_n)]^{-1/(\gamma_0+\gamma_1)}$
for \emph{ordinary sparse} measures.
\end{longlist}
The notion of ordinary and supersparse measures mentioned in (iii)
will be defined in Section~\ref{Sec-descript}. At a high level,
they refer to probablity measures that have geometrically sparse support
on $\Theta$, where the sparseness is characterized
in terms of parameters $\gamma_0$ and
$\gamma_1$, which are, respectively, analogous to the Hausdorff dimension
and the packing dimension that arise in fractal
geometry~\cite{Falconer,Garcia-etal-07}.

Our last main theorem
establishes the effect of ``borrowing strength'' of hierarchical modeling.
Suppose
that an i.i.d. $\nn$-sample $Y_{[\nn]}^0$ drawn from a mixture model
$Q_0*f$ is available,
where $Q_0 = Q_0^*\in\Pcal(\Theta)$ is unknown:
%
%
\begin{equation}
Y_{[\nn]}^0 | Q_0 \stackrel{\mathrm{i.i.d.}} {
\sim} Q_0*f.
\end{equation}
In a stand-alone setting $Q_0$
is endowed with a Dirichlet prior: $Q_0\sim\DP_{\alpha_0 H_0}$ for
some known $\alpha_0 > 0$ and nonatomic base measure $H_0 \in\Pcal
(\Theta)$.
Under mild conditions on the Dirichlet process mixture,
it can be shown that in Hellinger metric, the posterior probability
%
%
\begin{equation}
\label{Eqn-rate-standalone} \Pi_Q \bigl(h\bigl(Q_0*f,Q_0^**f
\bigr)\geq C( \log\nn/\nn)^{{1}/{(d+2)}} | Y_{[\nn]}^0 \bigr)
\longrightarrow0
\end{equation}
in $P_{Y^0_{[\nn]}|Q_0^*}$-probability for some constant $C>0$
(see \cite{Nguyen-11}).
Alternatively, suppose that
$Q_0$ is attached to the hierarchical Dirichlet process in
the same way as the $Q_1,\ldots,Q_m$, that is,
%
%
\begin{eqnarray}
\label{Eqn-HM-4intro} G \sim\DP_{\gamma H},\qquad Q_0, Q_1,
\ldots, Q_m | G \stackrel{\mathrm{i.i.d.}} {\sim} \DP_{\alpha G}.
\end{eqnarray}
Implicit in this specification, due to a standard property
of the Dirichlet, is the assumption that $Q_0$ shares
the same set of supporting atoms as $Q_1,\ldots,Q_m$, as they share
with the (latent) discrete base measure $G$.

Theorem~\ref{Thm-main-2} in Section~\ref{Sec-descript} establishes
the posterior concentration rate $\delta_{m,n,\nn}$ for the mixture
density $Q_0*f$, under the hierarchical model given by equation \eqref
{Eqn-HM-4intro},
as $\nn\rightarrow\infty$ and $m,n \rightarrow\infty$
at suitable rates.
%
Specifically, suppose that the true base measure $G_0$ has a
finite number of support points, if $m$ and
$n$ grow sufficiently fast relatively to $\nn$ so that the base measure
$G$ converges to $G_0$ at a sufficiently fast rate, then
the ``borrowing of strength'' from the $m\times n$ data set $\YYn^{[m]}$
to the inference about the data set $Y_{[\nn]}^0$ has a striking effect:
In particular, if $f$ is an ordinary smooth kernel density, we obtain
$\delta_{m,n,\nn} \asymp(\log\nn/\nn)^{1/2}$. If $f$ is a supersmooth
kernel density with smoothness $\beta> 0$, then $\delta_{m,n,\nn}
\asymp
(1/\nn)^{1/(\beta+2)}$. (The formal definition of smoothness
conditions is
given in Section~\ref{Sec-descript}.) These present
sharp improvements from nonparametric rate $(\log\nn/\nn)^{1/(d+2)}$
in equation \eqref{Eqn-rate-standalone}.
Thus, the hierarchical models are particularly beneficial to
groups of data with small sample sizes, as the convergence
of the latent variable further up in the hierarchy can be
translated into faster (e.g., parametric) rates of convergence
of these small-sample groups.
This appears to be the first result that establishes the benefits of
the latent hierarchy in a concrete manner.

\textit{Technical approach}.
The major part of the proof of the main theorems
lies in our attempt to understand the identifiability of the
Dirichlet base measure based on the marginal densities of the
data. This is achieved by establishing suitable inequalities
relating the three quantities:
(1)~a~Wasserstein distance between two base measures, $\Wr(G,G')$,
(2) a suitable notion of distance between
Dirichlet measures $\DP_{\alpha G}$ and $\DP_{\alpha'G'}$, and
(3) the variational distance or Kullback--Leibler divergence between
the densities
of $n$-vector $\YYn$,
which are obtained by integrating out the (latent) Dirichlet
process $Q$ that is distributed by Dirichlet measures $\DP_{\alpha G}$
and $\DP_{\alpha' G'}$.
In fact, the establishment of these inequalities takes up the most
space of this paper (Sections \ref{Sec-transport}, \ref{Sec-boundary}
and~\ref{Sec-contraction}).
To this end, we define a notion of optimal transport distance between Dirichlet
measures $\DP_{\alpha G}$ and $\DP_{\alpha' G'}$ (see equation
\eqref
{Eqn-W2-def}),
which is the optimal
cost of moving the mass of atoms lying in the support of measure $\DP
_{\alpha G}$
to that of $\DP_{\alpha' G'}$, where the cost of moving
from an atom (i.e., a~measure) $P_1 \in\Pcal(\Theta)$ to another
measure $P_2 \in\Pcal(\Theta)$
is again defined as a Wasserstein distance $\Wr(P_1,P_2)$ given by
equation \eqref{Eqn-W-def}.
In general, one can define distances of measures of measures and so on
in a recursive way. This provides means for comparing between Bayesian
hierarchies of random measures for an arbitrary number of hierarchy levels
(see Section~\ref{Sec-transport}).

In order to derive inequalities for the aforementioned distances,
our approach boils down to establishing the
existence of a subset of $\Pcal(\Theta)$ which can be used
to distinguish one Dirichlet measure from a class of Dirichlet measures.
Because we do not have direct access to the samples $Q_i$
of a Dirichlet measure,
only the estimates of such samples, the test set has to be robust.
By robustness, we require that the measure of a tube-set constructed along
the boundary of the test set be \emph{regular}, by which we mean
that it is possible to
control the rate at which such measure vanishes, as the radius in Wasserstein
metric of such tube-set tends to zero. Interestingly,
the precise vanishing rates are closely linked to the geometrically
sparse structure of the support of the true Dirichlet base measure.
These results are developed in Section~\ref{Sec-boundary}
and Section~\ref{Sec-contraction}.

The proof of Theorem~\ref{Thm-main-2} requires results concerning
the geometry of the support of a single Dirichlet measure.
Although the support of a Dirichlet measure is very large, that is,
the entire space $\Pcal(\Theta)$ (cf. \cite{Ferguson}),
we show that most of the mass
of a Dirichlet measure concentrates on a very small set as measured
by the covering number of Wasserstein balls defined on $\Pcal(\real^d)$.
Our result generalizes to higher dimensions
the behavior of tail probabilities chosen from
a Dirichlet measure on $\Pcal(\real)$ \cite{Doss-Sellke}.

\textit{Limitations of our results}.
The asymptotic results established in this paper
are distinguished by the nonstandard
roles of two quantities $m$ and $n$ simultaneously present
in the model. Although both determine the size of observed
data, they play asymmetric roles in the model hierarchy:
$m$~is the number of groups of data, and $n$ is the sample
size for each group.
When $n$ is fixed and $m$ increases, the concentration rates established
for marginal densities of $n$-vectors in Theorem~\ref{Thm-main-0}
are optimal up to some logarithmic terms in several settings. However,
when $n$ is allowed to increase, the rate gets worse.
For parametric models, the $\log n$ term
may be ignored. Unfortunately, for nonparametric models, the
presence of a polynomial quantity of $n$ in the numerator
may be suboptimal. Such presence of $n$ in the rate
is due to the fact that
the space of the marginal densities on $n$-vector
$\YYn$ data appears to get larger with $n$. This explanation appears
reasonable, but
we should be quickly reminded that the $n$ elements of $\YYn$ are in
fact exchangeable -- they carry a special dependence structure
among themselves. In short, having explained the role of $n$
in its appearance in the posterior concentration rate's upper bound,
we do not know whether this appearance is optimal. A more definitive
conclusion on the optimal
nature of convergence rates of the marginal density can
only be achieved by directly tackling a minimax theory of
density estimation for exchangeable sequences.
Such a theory is not available at the moment.

On the more difficult question regarding the inference of base measure $G$,
our result given by Theorem~\ref{Thm-main-1}
exhibits some notable weaknesses. First of all, the
posterior concentration rate \eqref{Eqn-concentration}
is meaningful only in the regime that both
$m$ and $n$ increase. The intuition behind our analysis for
$G$ is quite natural: as $n$ increases, one should get a better handle on
individual parameter $Q_i$ in each group. And with $m$
increasing as well, one should be able to improve the quality of
the inference of the base measure $G$ on the basis of the $Q_i$'s.
Unfortunately, if $n$ grows too fast relatively to~$m$,
the upper bound \eqref{Eqn-concentration} gets worse
(and eventually becomes useless).
Note that in this paper we are still
unable to establish posterior concentration behavior for $G$
in the case where $n$ is fixed, and $m$ grows (except the case
$n=1$).
Our present techniques are probably not powerful
enough to address this interesting and arguably more practical
asymptotic regime. The limitations seems to have
their roots in a decoupling technique employed in the
development of Theorem~\ref{Thm-contraction}
in Section~\ref{Sec-contraction},
which derives an upper bound for the Wasserstein distances of Dirichlet
base measures in terms of the corresponding marginal densities
on $n$-vector $\YYn$. These issues will be elaborated further
in the paper.

\textit{Organization of the paper.}
Section~\ref{Sec-descript} describes
the model setting and provides a full statement of the main theorems.
Section~\ref{sec-method} elaborates on the components of the proofs
and the tools that we develop. Section~\ref{Sec-transport}
defines transportation distances for hierarchies of random measures.
Section~\ref{Sec-boundary} analyzes regular boundaries of test
sets that arise in the support of various classes
of Dirichlet measures of interest.
Section~\ref{Sec-contraction} gives upper bounds
for Wasserstein distances of base measures. The proof of
Theorem~\ref{Thm-main-0} is given in Section~\ref{Sec-transport},
the proof of Theorem~\ref{Thm-main-1} is given later
in Section~\ref{Sec-contraction},
which draws from the machinery developed in Sections \ref
{Sec-transport}, \ref{Sec-boundary} and \ref{Sec-contraction}.
The proof of Theorem~\ref{Thm-main-2}
is given in Section~\ref{Sec-perturb}, which also draws on the
results on the geometry of the support of a single Dirichlet measure.

\textit{Notation}.
$\Wr$ denotes the $L_r$ Wasserstein distance.
$N(\varepsilon, \Gcal, \Wr)$ denotes the covering number of $\Gcal$
in metric $\Wr$. $D(\varepsilon, \Gcal, \Wr)$ is the packing number of
the same metric \cite{vanderVaart-Wellner-96}.
$\supp G$ denotes the support of probability measure $G$.
Several divergence functionals of probability densities are employed:
$K(p,q), h(p,q), V(p,q)$ denote
the Kullback--Leibler divergence, Hellinger and variational distance
between two densities $p$ and $q$ defined with respect to a measure
on a common space: $K(p,q) = \int p\log(p/q) $,
$h^2(p,q) = \frac{1}{2}\int(\sqrt{p} - \sqrt{q})^2$
and $V(P,Q) = \frac{1}{2}\int|p -q|$.
In addition, we define $K_2(p,q) = \int p[\log(p /q)]^2$,
$\chi(p,q) = \int p^2/q$. $A \lesssim B$ means $A \leq C\times B$ for
some positive constant $C$ that is either universal or specified
otherwise.
Similarly, for $A \gtrsim B$.

\section{Main theorems and tools}
\label{Sec-descript}
\subsection{Model setting and definitions}
Consider the following hierarchical probabilistic model:
%
%
\begin{eqnarray}
\label{Eqn-HM-A}&& G \sim\DP_{\gamma H},\qquad Q_1,\ldots,
Q_m | G \stackrel{\mathrm{i.i.d.}} {\sim} \DP_{\alpha G},
\\
\label{Eqn-HM-B} &&\YYn^{i}:= (Y_{i1},\ldots,Y_{in})
| Q_i \stackrel{\mathrm{i.i.d.}} {\sim} Q_i*f \qquad\mbox{for } i=1,
\ldots, m.
\end{eqnarray}
The relationship among quantities of interest can be
illustrated by the following diagram:
%

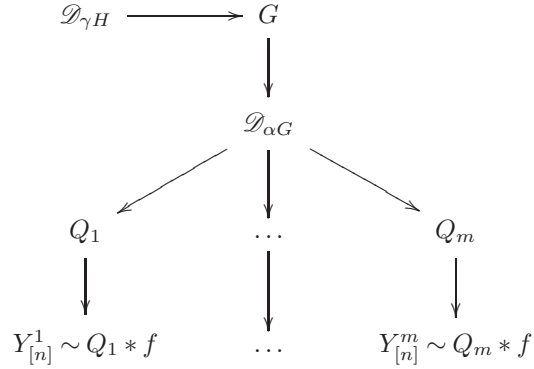
\[
\xymatrix@C=0.70cm @R=0.70cm{
\xymatrixcolsep{40pt}
\DP_{\gamma H} \ar[r] &G \ar[d] &\\
& \DP_{\alpha G} \ar[dl] \ar[d] \ar[dr]& \\
Q_1\ar[d] & \ldots\ar[d]& Q_m\ar[d] \\
\YYn^{1} \sim Q_1*f & \ldots&\YYn^{m} \sim Q_m*f }
\]
%
%


Dropping the index $i$,
$\YYn:=(Y_1,\ldots, Y_n)$ denotes
the generic i.i.d. random $n$-vector according to
the generic mixture density $Q*f$, where $Q$ is sampled from
Dirichlet measure $\DP_{\alpha G}$. The
marginal density of $\YYn$ takes the form:
%
%
\begin{equation}
\label{Eqn-pG} {p_{Y_{[n]}|G}}(\YYn) = \int\prod
_{j=1}^{n}Q*f(Y_{j})
\DP_{\alpha G}(\mathrm{d}Q).
\end{equation}
Given an $m\times n$ data set
$\YYn^{[m]} := (\YYn^1,\ldots,\YYn^m)$,
the posterior distribution of $G$ given $\YYn^{[m]}$ takes the form,
for any measurable $\BB\subset\Pcal(\Theta)$:
%
%
\begin{equation}
\label{Eqn-priorG} \Pi_G \bigl(G\in\BB|\YYn^{[m]} \bigr) =
\frac{\int_{\BB}\prod_{i=1}^{m}{p_{Y_{[n]}|G}}(\YYn^i)
\DP_{\gamma H} (\mathrm{d}G)}{
\int_{}\prod_{i=1}^{m}{p_{Y_{[n]}|G}}(\YYn^i) \DP
_{\gamma H}(\mathrm{d}G)}.
\end{equation}

There are three main theorems. The first is concerned with the
concentration behavior of the posterior distribution of marginal
density ${p_{Y_{[n]}|G}}$ given
the data $\YYn^{[m]}$, as $m \rightarrow\infty$, assuming that
the data is generated according to $G = G_0$ for some fixed
$G_0 \in\Pcal(\Theta)$. The second deduces the posterior contraction
of the base measure $G$, reposing upon that of ${p_{Y_{[n]}|G}}$.
The third theorem is
concerned with the concentration behavior of an individual
mixing measure $Q_i$ given the data.

\textit{Geometric sparseness conditions for $G_0$.} Our theory is
developed for a class of atomic base measure $G_0$. A simple example
is the case $G_0$ has a finite number of support
points. We also consider the case $G_0$ has infinite support,
which admits a geometrically sparse structure that we now define.

%
\begin{definition}
\label{Def-sparse}
Given $c_1 \in(0,1), c_2 > 0$ and a nonincreasing function
$K\dvtx\real_+\rightarrow\real_+$. A subset $S$ of metric space
$\Theta$ is $(c_1,c_2,K)$-sparse
if for any sufficiently small
$\delta> 0$ there is $\varepsilon\in(c_1\delta,\delta)$
according to which $S$ can
be covered by at most $K(\varepsilon)$ closed balls of radius
$\varepsilon$,
and every pair of such balls is separated by a distance at least
$c_2\varepsilon$.
\end{definition}

Probability measure $G_0$ is said to be sparse, if its support is a
$(c_1,c_2,K)$-sparse for a valid combination of
$c_1,c_2$ and $K$.
A \emph{gauge function} for a sparse measure $G_0$, denoted by
$\gauge\dvtx\real_+\rightarrow\real$, is defined
as the maximal function such that for each sufficiently small
$\varepsilon$,
there is a valid $\varepsilon$-covering specified
by the definition and that the $G_0$ measure on each of the covering
$\varepsilon$-balls is bounded from below by
$\gauge(\varepsilon)$. $\gauge$ is clearly a nondecreasing function.

We say $G_0$ is \emph{supersparse} with
nonnegative parameters $(\gamma_0,\gamma_1)$, if
function $K$ satisfies $K(\varepsilon) \lesssim[\log(1/\varepsilon
)]^{\gamma_0}$,
and function $\gauge$ satisfies
$\gauge(\varepsilon) \gtrsim[\log(1/\varepsilon)]^{-\gamma_1}$.
$G_0$ is \emph{ordinary sparse} with parameters $(\gamma_0,\gamma
_1)$ if
$K(\varepsilon) \lesssim(1/\varepsilon)^{\gamma_0}$,
and $\gauge(\varepsilon) \gtrsim\varepsilon^{\gamma_1}$.

\textit{Examples}.
If $\Theta= [0,1]$ and $S = \{1/2^k| k \in\Nat, k\geq1\} \cup\{0\}$,
then $S$ is $(c_1,c_2,K)$-sparse with $c_1=1/2, c_2= 2$ and
$K(\varepsilon) = \log(1/2\varepsilon)/\log2$. If $S$ is the
support of
$G_0$, and $G_0(\{1/2^k\}) \propto k^{-\gamma_1}$ for any $k\in\Nat$
and some $\gamma_1 > 1$,
then $G_0$ is clearly a supersparse measure with parameters $\gamma_0
= 1$
and $\gamma_1$. Ordinary sparse measures as we defined typically arise
in fractal geometry \cite{Falconer}, where parameter $\gamma_0$ is
analogous to
the Hausdorff dimension of a set, while $\gamma_1$ is analogous to the packing
dimension (see, e.g., \cite{Garcia-etal-07}). Now,
if $\Theta= [0,1]$ and $S$ is the classical Cantor set,
then $S$ is $(c,K)$-sparse with $c_1=1/3, c_2 = 2$
and $K(\varepsilon) = \exp[\log(1/2\varepsilon)\log2/\log3]$.
Set $S$ has Hausdorff dimension equal $\gamma_0 = \log2/\log3$.
Let $G_0$ be the $\gamma_0$-dimension Hausdorff measure on set $S$,
then $G_0$ is ordinary sparse with $\gamma_0 = \gamma_1 = \log2/\log3$.

\textit{Conditions on kernel density $f$.}
The main theorems in this paper are established independently
of the specific choices of kernel density $f$ except
some minor assumptions (A1), (A2) in the sequel.
However, to obtain concrete rates in $m$ and $n$,
we will make additional assumptions on the smoothness of $f$ when needed.
Such assumptions are chosen mainly
so we can make use of the concrete rates of demixing in a deconvolution
problem, that is, the convergence rate of a point estimate of a mixing
measure $Q$
given an i.i.d. sample from the mixture density $Q*f$.

For that purpose, $f$ is a density function on $\real^d$ that is
symmetric around 0, that is,
$f(x|\theta) := f(x-\theta)$ such that $\int_{B} f(x) \,\mathrm{d}x = \int_{-B} f(x)\,\mathrm{d}x$
for any Borel set $B \subset\real^d$.
In addition, the Fourier transform of $f$ satisfies
$\tilde{f}(\omega) \neq0$ for all $\omega\in\real^d$.
We say $f$ is \emph{ordinary smooth} with parameter $\beta>0$ if
$\int_{[-1/\delta,1/\delta]^d} \tilde{f}(\omega)^{-2} \,\mathrm
{d}\omega
\lesssim(1/\delta)^{2d\beta}$ as $\delta\rightarrow0$.
Say $f$ is \emph{supersmooth} with parameter $\beta> 0$
if $\int_{[-1/\delta,1/\delta]^d} \tilde{f}(\omega)^{-2} \,\mathrm
{d}\omega
\lesssim\exp(2d\delta^{-\beta})$ as $\delta\rightarrow0$.
These definitions are somewhat simpler and more general than
what is employed in \cite{Nguyen-11}.
Depending on the form of $f$, it was
shown by \cite{Nguyen-11} that
there is a strictly increasing
function $\Psi\dvtx\real_+ \rightarrow\real_+$ that there holds
%
%
\begin{equation}
\label{Eqn-Psi-bound} W_2 \bigl(Q,Q' \bigr) \lesssim\Psi
\bigl(V \bigl(Q*f,Q'*f \bigr) \bigr)
\end{equation}
for any pair $Q,Q'\in\Pcal(\Theta)$, provided that
$\Theta$ is a bounded subset of $\real^d$, and $W_2(Q,Q)$
is sufficiently small. In particular,
if $f$ is ordinary smooth with parameter $\beta$, then
$\Psi(u) = u^{1/(2+\beta d')}$ for any $d' > d$.
If $f$ is supersmooth, then
$\Psi(u) = (-\log u)^{-1/\beta}$ (cf. Theorem~2 of \cite{Nguyen-11}).

\subsection{Main theorems}

The following list of assumptions are required throughout the paper:
\begin{longlist}[(A1)]
\item[(A1)] For some $r\geq1, C_1 > 0$,
$h(f(\cdot|\theta), f(\cdot|\theta')) \leq C_1\|\theta-\theta'\|^{r}$
and
$K(f(\cdot|\theta), f(\cdot|\theta')) \leq C_1\|\theta-\theta'\|^{r}\
\forall\theta, \theta' \in\Theta$.

\item[(A2)] There holds
$M = \sup_{\theta,\theta' \in\Theta} \chi(f(\cdot|\theta
),f(\cdot
|\theta')) < \infty$.

\item[(A3)] $H \in\Pcal(\Theta)$ is nonatomic, and
for some constant $\eta_0 > 0$, $H(B) \geq\eta_0 \varepsilon^d$ for
any closed ball $B$ of radius $\varepsilon$.
\end{longlist}
It is simple to observe that
(A1) holds for $r=2$ for the Gaussian kernel density $f$,
and holds for $r=1$ for almost all standard kernel densities in the
modeling literature (Laplace, Cauchy, Gamma, etc.).
(A2) holds naturally for most choices of kernel densities, as long as
$\Theta$ is bounded. (A3) is often satisfied by
almost all (noninformative) prior choices made in practice.

We are ready to state the first theorem, which
establishes the posterior concentration of the
marginal density of $n$-vector $\YYn$ under the above assumptions.

%
\begin{theorem}
\label{Thm-main-0}
Let $\Theta$ be a bounded subset of $\real^d$ and
$G_0 \in\Pcal(\Theta)$. Given assumptions \textup{(A1)--(A3)},
parameters $\alpha> 0, \gamma> 0$ and $H \in\Pcal(\Theta)$ are known.
Let $m$ tend to infinity, while $n$ can be either fixed to a constant,
or $n$ tending to infinity at a rate relatively to $m$.
Then there is a large constant $C$ independent of both $m$ and $n$ such that
the posterior induced by the model of equations \eqref{Eqn-HM-A}
and \eqref{Eqn-HM-B}
satisfies
\[
\Pi_G \biggl(h({p_{Y_{[n]}|G_0}},
{p_{Y_{[n]}|G}}) \geq C \biggl[\frac{n^{3d}\log(mn)}{m} \biggr]^{1/(2d+2)} \Big|
\YYn^{[m]} \biggr) \longrightarrow0
\]
in $\PGz^m$-probability. Moreover,
\begin{longlist}[(ii)]
\item[(i)] If $f$ is a Gaussian kernel with a fixed variance, then the
rate is improved to
\[
\epsmn= \biggl[\frac{n^{2d}(\log m)^{2d+1}\log n}{m} \biggr]^{1/2}.
\]
\item[(ii)] If $G_0$ has a finite \emph{and} known number of support points,
then the rate is improved to
\[
\epsmn= \biggl[\frac{\log(mn)}{m} \biggr]^{1/2}.
\]
\end{longlist}
\end{theorem}

\begin{remarks*}
1. When $n$ is fixed, the dependence of the rate on $n$ carries
no consequence. The theorem establishes in several cases that
the concentration rate with respect to $m$ is the
optimal $m^{-1/2}$ up to a logarithmic quantity. This includes the
parametric case
(i.e., $G_0$ is assumed to have a known finite number of support points).
But the much more interesting case is
when one uses a Gaussian density kernel $f$, despite the possibility
that $G_0$ may still have infinite support. In the general setting, where
almost nothing is assumed of $G_0$ and $f$ (except relatively mild
assumptions in (A1)--(A3)), the nonparametric rate of $m^{-1/(2d+2)}$
appears quite natural.

2. When $n$ is allowed to vary along with $m$,
increasing $n$ has the effect of worsening our upper bound for
the posterior concentration rate.
An explanation for this phenomenon is that as $n$ gets large, the marginal
density ${p_{Y_{[n]}|G}}$ may become more degenerate.
More concretely, in the calculations that we shall
present later, the (estimate of the) entropy of the space of
marginal densities $\{{p_{Y_{[n]}|G}}| G \in\Pcal(\Theta
)\}$ under Hellinger
metric is shown to increase with $n$ (cf. Lemma~\ref{Lem-entropy}).
Only in the case of a parametric model (i.e., the number of support
points of $G_0$ is known) do we observe that the effect of $n$ is
the negligible $(\log n)$. We do not know whether the presence of
$n$ in the rate's numerator is optimal -- a definitive answer regarding
the optimality of these rates may be
settled by a
minimax analysis, which is beyond the scope of this paper.
\end{remarks*}

Next, we turn to the posterior concentration of the base measure
$G$ per se. An easy bound can be deduced for the case $n=1$
from Theorem~\ref{Thm-main-0}. Due the basic property of the
Dirichlet measure that $\int Q(\mathrm{d}\theta) \DP_{\alpha G}(\mathrm{d}Q)
= G(\mathrm{d}\theta)$, and by an application of Fubini's theorem,
the marginal density for a single data point takes the form:
\begin{eqnarray*}
\pGone(\YYone) & = & \int\int f(Y_1-\theta)Q(\mathrm{d}\theta)
\DP_{\alpha G}(\mathrm{d}Q)
\\
& = & \int f(Y_1 - \theta) G(\mathrm{d}\theta) = G*f(Y_1).
\end{eqnarray*}
Provided that all conditions stated in Theorem~\ref{Thm-main-0} hold,
so that the posterior concentrate rate $\varepsilon_{m1}
\asymp[\log(m)/m]^{1/(2d+2)}$ is attained
for the marginal density $\pGone$, as $n=1$ and $m\rightarrow\infty$.
Combining this concentration rate with equation \eqref{Eqn-Psi-bound}
gives the
following:
\[
\Pi_G \bigl(W_2(G,G_0) \leq\Psi(
\varepsilon_{m1}) |Y_{[1]}^{[m]} \bigr)
\longrightarrow1
\]
in $P^m_{\YYone|G_0}$-probability, as $m\rightarrow\infty$.

Unfortunately, we do not know how to extend this bound to the
case where $n$ is fixed to a constant greater than 1.
In the following, we shall work in a regime where both
$m$ and $n = n(m)$ tend to infinity.
Let $(\varepsilon_n,\delta_n)_{n\geq1}$ be two nonnegative vanishing
sequences,
where $\delta_n = \Psi(\varepsilon_n)$
such that $\exp-n\varepsilon_n^2 = \mathrm{o}(\delta_n)$
and that the following holds:
for any $Q\in\Pcal(\Theta)$, there exists a point estimate
$\hatQ_n$ given an $n$-i.i.d. sample
from the mixture distribution $Q*f$, such that the following
inequality holds:
%
%
\begin{equation}
\label{Eqn-sample-bound} \Prob \bigl(W_2(\hatQ_n,Q) \geq
\delta_n \bigr) \leq5\exp \bigl(-c n\varepsilon_n^2
\bigr),
\end{equation}
where constant $c$ is universal, the probability measure $\Prob$
is given by the mixture density $Q*f$. We refer to $\delta_n$
as the demixing rate. The exact nature of
$(\varepsilon_n,\delta_n)$ is not of concern at this point.
In addition, define
\[
{\alpha^*}:= \alpha\inf_{\theta\in\supp G_0} G_0 \bigl(\{\theta
\} \bigr).
\]
Note that ${\alpha^*}> 0$ if $G$ has finite support, and ${\alpha^*}= 0$
otherwise.

%
\begin{theorem}
\label{Thm-main-1}
Let $\Theta$ be a bounded subset of $\real^d$ and $G_0 \in\Pcal
(\Theta
)$. Given
assumptions \textup{(A1)}--\textup{(A3)}, parameters $\alpha\in(0,1],
\gamma> 0$ and $H \in\Pcal(\Theta)$ are known.
Then, as $m \rightarrow\infty$ and $n = n(m) \rightarrow\infty$,
there is a sequence $\epsmn$ and $\Delta_n$ dependent on $m$ and $n$
such that
under the model given equations~\eqref{Eqn-HM-A} and \eqref{Eqn-HM-B},
there holds:
\[
\Pi_G \bigl(W_1(G,G_0) \leq C (\epsmn+
\Delta_n) | \YYn^{[m]} \bigr) \longrightarrow1
\]
in $\PGz^m$-probability for a large constant $C$ independent of
$m$ and $n$.
In particular,
$\epsmn$ is any posterior concentration rate for the marginal
densities such as the ones established by Theorem~\ref{Thm-main-0}.
Regarding the nature of $\Delta_n$,
\begin{longlist}[(iii)]
\item[(i)] If $G_0$ has finite (but unknown) number of support points, then
\[
\Delta_n \asymp\delta_n^{{\alpha^*}/({\alpha^*}+1)}.
\]

\item[(ii)] If $G_0$ has infinite and supersparse support with parameters
$(\gamma_0,\gamma_1)$, then
\[
\Delta_n \asymp\exp- \bigl[\log(1/\delta_n)
\bigr]^{1/(1\vee\gamma_0 +\gamma_1)}.
\]

\item[(iii)] If $G_0$ has infinite and ordinary sparse support with parameters
$(\gamma_0,\gamma_1)$, then
\[
\Delta_n \asymp \bigl[\log(1/\delta_n)
\bigr]^{-1/(\gamma_0+\gamma_1)}.
\]
\end{longlist}
\end{theorem}

\begin{remarks*}
1. Section~\ref{Sec-contraction} establishes
the existence of a point estimate which admits the
finite-sample probability bound \eqref{Eqn-sample-bound}.
In particular, $\varepsilon_n$ is given as follows:
$\varepsilon_n \asymp(\log n/n)^{r/2d}$, if $d > 2r$;
$\varepsilon_n \asymp(\log n/n)^{r/(d+2r)}$ if $d<2r$, and
$\varepsilon_n \asymp(\log n)^{3/4}/n^{1/4}$ if $d=2r$. Constant $r$
is from assumption (A1). The rate of demixing $\delta_n$ is determined
according to an additional condition on
the smoothness of the kernel density $f$:
\begin{longlist}[(a)]
\item[(a)] If $f$
is ordinary smooth with parameter $\beta> 0$,
then $\delta_n = \varepsilon_n^{{1}/{(2+\beta d')}}$ for any
$d' > d$.
\item[(b)] If $f$ is supersmooth with parameter $\beta> 0$, then
$\delta_n = [-\log\varepsilon_n]^{-1/\beta}$.
\end{longlist}

2. In the parametric case,
the number of support points of $G_0$ is $k < \infty$ and
$k$ is known, $H$ is taken to be a probability measure with
$k$ support points. Then we obtain
the following parametric rate of posterior concentration for a
finite admixture model for continuous data:
\[
\epsmn+ \Delta_n = \bigl[\log(mn)/m \bigr]^{1/2} +
\delta_n^{{\alpha^*}}.
\]
Under identifiability
conditions for kernel density $f$, such as those considered
by \cite{Nguyen-11} (Theorem~1), one has
$\varepsilon_n = (\log n)n^{-1/2}$ and
$\delta_n = \varepsilon_n^{1/2} = (\log n)^{1/2}n^{-1/4}$.
Finite admixtures for categorical data exhibit a quite
different kind of geometry, and were investigated
in \cite{Nguyen-admixture}.

3. The above theorem establishes that the posterior concentration rate
is bounded from above by two quantities $\epsmn$ and $\Delta_n$.
The former captures the contraction of the marginal density
of observed data, while the latter captures the demixing
(deconvolution) aspect of each individual mixing measure
$Q_i$.
It is natural to expect that $\Delta_n \gg\delta_n$,
to account for the fact that the mixing measures $Q_i$ are not observed
directly. It is interesting how quantity $\Delta_n$ depends on
the geometric sparsity of the support of the true base measure $G_0$:
as
$G_0$ becomes less sparse, $\Delta_n$ gets slower:
\[
\delta_n \ll\delta_n^{{\alpha^*}} \ll
\delta_n^{{\alpha^*}
/({\alpha^*}+1)} \ll\exp- \bigl[\log(1/\delta_n)
\bigr]^{1/(1\vee\gamma_0 +\gamma_1)} \ll \bigl[\log(1/\delta _n)
\bigr]^{-1/(\gamma_0+\gamma_1)}.
\]
\end{remarks*}

Our final main result is about the posterior concentration
behavior of the latent mixing measures $Q_i$, as the base measure $G$ is
integrated out, and the amount of data increases. For the ease
of presentation, we isolate a particular mixing measure to be denoted by
$Q_0$, and we shall assume that
$Q_0$ is attached to the hierarchical Dirichlet process in
the same way as the $Q_1,\ldots,Q_m$, that is,
%
%
\begin{eqnarray}
\label{Eqn-HM-4} G \sim\DP_{\gamma H},\qquad Q_0, Q_1,
\ldots, Q_m | G \stackrel{\mathrm{i.i.d.}} {\sim} \DP_{\alpha G}.
\end{eqnarray}
%
%
Suppose that an i.i.d. $\nn$-sample $Y_{[\nn]}^0$ drawn from a mixture
model $Q_0*f$ is available,
where $Q_0 = Q_0^*\in\Pcal(\Theta)$ is unknown:
%
%
\begin{equation}
\label{Eqn-HM-5} Y_{[\nn]}^0 | Q_0
\stackrel{\mathrm{i.i.d.}} {\sim} Q_0*f.
\end{equation}
In addition, as before, $m\times n$ data set is available:
%
%
\begin{equation}
\label{Eqn-HM-6} \YYn^{i}:= (Y_{i1},\ldots,Y_{in})
| Q_i \stackrel{\mathrm{i.i.d.}} {\sim} Q_i*f \qquad\mbox{for } i=1,
\ldots, m.
\end{equation}
The relationship among quantities of interest is
illustrated by the following diagram:

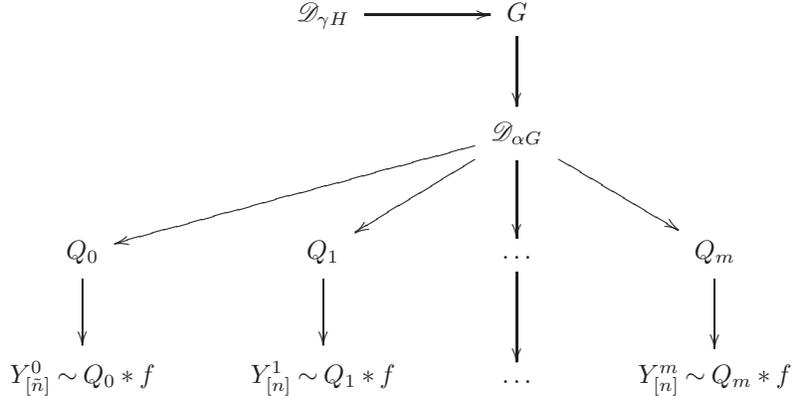
\[
\xymatrix{
 & \DP_{\gamma H}\ar[r] & G\ar[d] & \\
& & \DP_{\alpha G}\ar[dll]\ar[dl]\ar[d]\ar[dr] & \\
Q_0\ar[d] & Q_1\ar[d] & \ldots\ar[d]& Q_m \ar[d]\\
Y_{[\nn]}^0 \sim Q_0*f & \YYn^{1} \sim Q_1*f & \ldots&\YYn^{m} \sim
Q_m*f \\
}
%
%
\]


The following theorem shows that the posterior distribution
$\Pi(Q_0|Y_{[\nn]}^0,\YYn^{[m]})$, defined with respect to
specifications \eqref{Eqn-HM-4},
\eqref{Eqn-HM-5} and \eqref{Eqn-HM-6}, concentrates most its mass toward
$Q_0^*$, as $n, m$ and $\nn\rightarrow\infty$ appropriately.
The intuition for this result is rather simple. As the data size
$m\times n$
grows appropriately, the posterior distribution for base measure $G$
concentrates around the true $G_0$, which shall be assumed to be
a discrete measure with a finite, but unknown number of support point.
This benefits
the inference of density $Q_0*f$. Indeed,
the (conditional) Dirichlet prior on the mixing measure $Q_0$
(given the $m\times n$ data) can be shown to be very thick, due to
the fact that its base measure $G_0$ is conditionally close to a
measure with a finite number of support points. In addition,
one can identify subsets of the support of the (conditional)
Dirichlet prior for $Q_0$ which take up most of its probability mass,
while remaining small in size, as evaluated by the entropy/covering number.
A combination of these two
facts result in very favorable posterior concentration for the
marginal density $Q_0*f$. In fact, the rates become parametric,
as they are independent of the parameter dimensionality $d$.
By contrary, if we do not have the concentration of base measure $G$,
there is very little control of the space over which $Q_0$ may vary. As a
result, one can only establish the standard nonparametric rate of convergence
under general conditions.

A complete statement of the theorem is the following.
Motivated by the conclusion of Theorem~\ref{Thm-main-1} we shall
assume that
the posterior distribution of $G$ concentrates at a certain
rate $\delmn$ toward the true base measure $G_0$, which
is now assumed to have a finite (but unknown) number of support points.
This concentration behavior can in turn
be translated to a sharp concentration behavior for the mixture density
$Q_0*f$. 

%
\begin{theorem}
\label{Thm-main-2}
Let $\Theta$ be a bounded subset of $\real^d$, $G_0,Q_0^* \in\Pcal
(\Theta)$.
Suppose that assumptions \textup{(A1)} and \textup{(A2)} hold for some $r \geq1$.\vadjust{\goodbreak}
Given parameters $\alpha\in(0,1], \gamma> 0$, and
$H\in\Pcal(\Theta)$ known. Assume further that:
\begin{longlist}[(a)]
\item[(a)] $G_0$ has $k < \infty$ support points in $\Theta$;
$Q_0^*\in\Pcal(\Theta)$ such that $\supp Q_0^*\subseteq
\supp G_0$.
\item[(b)] For each $\nn$, there is a net $\delmn= \delmn(\nn)
\downarrow0$
indexed by $m,n$ such that under
the model specifications \eqref{Eqn-HM-4}, \eqref{Eqn-HM-5} and
\eqref
{Eqn-HM-6},
there holds:
$\Pi_G (W_1(G,G_0) \geq
C \delmn| \YYn^{[m]}, Y_{[\nn]}^0)
\longrightarrow0$
in $\PGz^m\times P_{Y_{[\nn]}^{0}|Q_0^*}$-probability,
as $m \rightarrow\infty$ and $n=n(m) \rightarrow\infty$ at a suitable
rate with respect to~$m$.
Here, $C$ is a constant independent of $\nn, m, n$.
\end{longlist}
Then, as $\nn\rightarrow\infty$ and then $m$ and $n=(m)\rightarrow
\infty$, we have
\[
\Pi_Q \bigl( h\bigl(Q_0*f,Q_0^**f
\bigr) \geq \delta_{m,n,\nn} |Y_{[\nn
]}^0,
\YYn^{[m]} \bigr) \longrightarrow0
\]
in $P_{Y_{[\nn]}^0|Q_0^*}\times\PGz^m$-probability, where
the rates $\delta_{m,n,\nn}$ are given as follows:
\begin{longlist}[(iii)]
\item[(i)] $\delta_{m,n,\nn} \asymp(\log\nn/\nn
)^{1/(d+2)}+\delmn
^{r/2}\log(1/\delmn)$.
\item[(ii)] $\delta_{m,n, \nn} \asymp(\log\nn/\nn)^{1/2}$ if $f$ is
ordinary smooth
with smoothness $\beta>0$, and $n$ and $m$ grow sufficiently fast so that
$\delmn$ is sufficiently small relatively to $\nn$ (see details in the
remarks below).
%
\item[(iii)] $\delta_{m,n,\nn} \asymp(1/\nn)^{1/(\beta+2)}$, if
$f$ is supersmooth with smoothness $\beta> 0$, $n$ and $m$ grow
sufficiently fast so that $\delmn$ is sufficiently small relatively to
$\nn$.\vadjust{\goodbreak}
\end{longlist}
\end{theorem}

\begin{remarks*}
%
%
1. Condition (a) that $\supp Q_0^*\subseteq\supp G_0$
motivates the incorporation of mixture distribution
$Q_0*f$ into the Bayesian hierarchy\vadjust{\goodbreak} as specified by equation \eqref{Eqn-HM-4}.
According to the model,
$Q_0$ shares the same supporting atoms
with $Q_1,\ldots, Q_m$, as they all inherit from random base measure $G$.
Note also that the condition on the posterior of $G$ as
stated in (b) is closely related to but nonetheless different from the
conclusion reached
by Theorem~\ref{Thm-main-1}, due to the additional conditioning
on $Y_{[\nn]}^0$. This condition may be proved directly under
additional assumptions on $Q_0^*$ and $G_0$, by a technically cumbersome
(but conceptually simple) modification
of the proof of Theorem~\ref{Thm-main-1}. We avoid this unnecessary
complication
as it is not central to the main
message of the present theorem. 

2. In the statement of part (ii), $m$ and $n$ are required to grow
at a rate so that
$\delmn\lesssim\nn^{-(\alpha+k+M_0)} (\log\nn)^{-(\alpha+k-2)}$,
for some constant $M_0 > 0$ depending only on $d,k,\beta$ and
$\operatorname{diam}
(\Theta)$.
In part (iii), we require
$\delmn\lesssim\nn^{-2(\alpha+k)/(\beta+2)} (\log\nn)^{-2(\alpha+k-1)}
\exp(-4\nn^{\beta/(\beta+2)})$.

3. To appreciate the statistical content of this theorem, recall
a stand-alone setting in which $Q_0$
is endowed with an independent Dirichlet prior: $Q_0\sim\DP_{\alpha_0
H_0}$ for
some known $\alpha_0 > 0$ and nonatomic base measure $H_0 \in\Pcal
(\Theta)$.
Combining with the model specification expressed by~\eqref{Eqn-HM-5}, we
obtain the posterior distribution for mixture density $Q_0*f$,
which admits the following concentration behavior under some mild
conditions (cf. \cite{Nguyen-11}): 
%
%
\begin{equation}
\Pi_Q \bigl(h\bigl(Q_0*f,Q_0^**f
\bigr)\geq(\log\nn/ \nn)^{{1}/{(d+2)}} | Y_{[\nn]}^0 \bigr)
\longrightarrow0
\end{equation}
in $P_{Y^0_{[\nn]}|Q_0^*}$-probability.\vspace*{-1pt}
Now, the rate in the above display
should be compared to the general rate given by claim (i)
of Theorem~\ref{Thm-main-2}:
$(\log\nn/\nn)^{1/(d+2)} +
\delmn^{r/2}\log(1/\delmn)$.
The extra quantity $\delmn^{r/2}\log(1/\delmn)$ can be viewed
as the general ``overhead cost'' for maintaining the
latent hierarchy involving the random Dirichlet prior $\DP_{\alpha G}$
in the hierarchical model.

4. Claims (ii) and (iii) demonstrate
the benefits of hierarchical modeling for groups of data with relatively
small sample size: when $n \gg\nn$ (and $m = m(n) \rightarrow\infty$
suitably) so that $\delmn$
is sufficiently small, we obtain parametric rates for
the mixture density $Q_0*f$: $(\log\nn/\nn)^{1/2}$ for ordinary smooth
kernels, and $(1/\nn)^{1/(\beta+2)}$ for
supersmooth kernels. This is a sharp improvement over the
standard rate $(\log\nn/\nn)^{1/(d+2)}$ one would get for fitting
a stand-alone mixture model $Q_0*f$ using a Dirichlet process
prior. Technically, this improvement
is due to the confluence of two factors:
By attaching $Q_0$ to the Bayesian hierarchy
one is able to exploit the assumption that random measure $Q_0$ shares the
same supporting atoms as the random base measure $G$.
This is translated to a favorable level of thickness of the
conditional prior for $Q_0$ (given the $m\times n$ data $\YYn^{[m]}$),
as measured by small Kullback--Leibler neighborhoods.
The second factor is due to our new construction of a sieves
(subsets of) $\Pcal(\Theta)$ over which the Dirichlet measure
concentrates most its mass on, but which have suitably small entropy
numbers. These details will be elaborated in Section~\ref{Sec-perturb}.
\end{remarks*}

Summarizing our results:
Theorem~\ref{Thm-main-0} establishes posterior concentration of the
marginal densities
generating the observed data, while Theorem~\ref{Thm-main-1}
establishes posterior concentration
of the latent Dirichlet base measure in a hierarchical setting.
Theorem~\ref{Thm-main-2} demonstrates
dramatic gains in the efficiency of statistical
inference of individual groups of data with relatively small
sample size. For groups with relatively large sample size, the
concentration rate appears to be weaken due to the overhead\vadjust{\goodbreak} of
maintaining the latent hierarchy.
This quantifies the effects of ``borrowing of
strength'', from large groups of data to smaller groups.
This is arguably a good virtue of hierarchical models:
it is the populations with smaller sample sizes that need improved
inference the most. 

\subsection{Method of proof}
\label{sec-method}
The major part of the proof of Theorem~\ref{Thm-main-0}
and \ref{Thm-main-1} lies
in our attempt to establish the relationship
between the three important quantities:
(1) a Wasserstein distance between two base measures, $\Wr(G,G')$,
(2) a suitable notion of distance between
Dirichlet measures $\DP_{\alpha G}$ and $\DP_{\alpha G'}$, and
(3) the variational distance/Kullback--Leibler
divergence between the marginal densities of $n$-vector $\YYn$,
which are obtained by integrating out the mixing measure $Q$, 
which is a Dirichlet process distributed by $\DP_{\alpha G}$
and $\DP_{\alpha G'}$, respectively.
The link from $G$ (resp., $G'$) to the induced $\PG$ (resp., $P_{Y_{[n]}|G'}$)
is illustrated by the following diagram:

\[
\xymatrix{
G \ar[r]\ar@{-}[d]& \DP_{\alpha G}\ar[r]\ar@{-}[d] & Q\ar[r]
& Q*f\ar[r] & \YYn\ar@{-}[d]\\
\Wr(G,G')\ar@{-}[d] & \Wr(\DP_{\alpha G},\DP_{\alpha G'})\ar@{-}[d]
 & & & V(\PG,P_{Y_{[n]}|G'})\ar@{-}[d]
\\
G'\ar[r] & \DP_{\alpha G'}\ar[r] & Q'\ar[r] & Q'*f\ar[r] & \YYn}
%
%
\]


In order to establish the relationship among the aforementioned
distances, we need to investigate
the geometry of the support of individual Dirichlet measures,
and the geometry of test sets that arise when a given
Dirichlet measure is tested (discriminated) against a large class of
Dirichlet measures. This study forms the bulk of
the paper in Section~\ref{Sec-transport}, Section~\ref{Sec-boundary}
and Section~\ref{Sec-contraction}.

\textit{Transportation distances for Bayesian hierarchies}.
To begin, in Section~\ref{Sec-transport} we
develop a general notion of transportation distance of
Bayesian hierarchies of random measures. This notion plays a
fundamental role in our theory, and we believe is also of
independent interest. Using transportation distances,
it is possible to compare between not only
two probability measures defined on $\Theta$,
but also two probability measures on the
space of measures on $\Theta$, and so on. Transportation
distances are natural for comparing between Bayesian hierarchies,
because the geometry of the space of support of measures is
inherited directly in the definition of the transportation distances
between the measures. In particular, $\Wr(\DP_{\alpha G},
\DP_{\alpha G'})$ is defined as the Wasserstein distance on the Polish
space $\Pcal(\Pcal(\Theta))$, by inheriting the Wasserstein
distance on the Polish space of measures $\Pcal(\Theta)$. (The
notation $\Wr$ is reused as a harmless abuse of notation.) It
can be shown that 
\[
\Wr(\DP_{\alpha G},\DP_{\alpha' G'}) \geq\Wr \bigl(G,G'
\bigr).
\]
The above inequality holds generally if $\DP_{\alpha G}$ and $\DP
_{\alpha' G'}$ are
replaced by any pair of probability measures on $\Pcal(\Theta)$ that
admit a suitable notion of mean measures $G$, and $G'$, respectively.
Moreover, the Dirichlet measures allow a remarkable identity:
when $\alpha= \alpha'$, we have
\[
\Wr(\DP_{\alpha G},\DP_{\alpha G'}) = \Wr \bigl(G,G'
\bigr).
\]
Repeated applications of Jensen's inequality yield
the following upper bound for the KL divergence:\footnote{Within
this subsection, the details on the constants
underlying $\lesssim$ and $\gtrsim$ are omitted for the
sake of brevity.}
\[
h^2(\PG,P_{Y_{[n]}|G'}) \leq K(\PG,{p_{Y_{[n]}|G'}})
\lesssim n \Wr^r(\DP_{\alpha G},\DP_{\alpha G'}) = n
\Wr^r \bigl(G,G' \bigr).
\]
%

\textit{Bounds on Wasserstein distances}.
The most demanding part of the paper lies in establishing
an upper bound of the Wasserstein distance $\Wr(G,G')$ in terms
of the variational distance $V({p_{Y_{[n]}|G}},
{p_{Y_{[n]}|G'}})$. This is ultimately achieved by
Theorem~\ref{Thm-contraction} in Section~\ref{Sec-contraction},
which states that for a fixed $G\in\Pcal(\Theta)$
and \emph{any} $G'\in\Pcal(\Theta)$,\vspace*{3pt}
%
%
\begin{eqnarray}
\label{Eqn-overview-1}&& \Wr^r \bigl(G,G' \bigr)
\nonumber
\\[-8pt]
\\[-8pt]
\nonumber
&&\qquad \lesssim V(\PG,
P_{Y_{[n]}|G'})
+ A_n \bigl(G,G' \bigr),\vspace*{3pt}
\end{eqnarray}
where $A_n(G,G')$ is a quantity that tends to 0 as $n\rightarrow\infty$.
The rate at which $A_n(G,G')$ tends to zero depends only on
the geometrically sparse structure of $G$, not $G'$.
The proof of this result hinges on the existence of
a suitable set $\BB_n \subset\Pcal(\Theta)$ measurable with
respect to (the sigma algebra induced by) the observed variables
$\YYn$, which can then be used to distinguish $G'$ from~$G$,
in the sense that\vspace*{3pt}
%
%
\begin{eqnarray}
\label{Eqn-overview-2} &&\Wr^r \bigl(G,G' \bigr)
\nonumber
\\[-8pt]
\\[-8pt]
\nonumber
&&\qquad \lesssim
P_{Y_{[n]}|G'}( \BB_n) - \PG(\BB_n) + A_n
\bigl(G,G' \bigr).\vspace*{3pt}
\end{eqnarray}
We develop two main lines of attack to arrive at a construction of $\BB_n$.

First, we establish the existence of a point estimate for the mixing
measure on the basis of the observed $\YYn$. Moreover, such point
estimates have to admit a finite-sample probability bound of the
following form: given $\YYn\sim Q*f$, there exist a point estimate
$\hatQ_n$ such that under the $Q*f$ probability, there holds\vspace*{3pt}
\[
\Prob \bigl(\Wr(\hatQ_n,Q) \geq\delta_n \bigr) \lesssim
\exp-n\varepsilon_n^2,\vspace*{3pt}
\]
where $\delta_n$ and $\varepsilon_n$ are suitable vanishing sequences.
These finite-sample bounds are presented in Section~\ref{Sec-contraction}.
The existence of $\hatQ_n$ will then be utilized in the construction
of a suitable set $\BB_n$. In particular, one may pretend to
have direct observations from the Dirichlet measures
to construct the test sets, with a possible loss of accuracy
captured by the demixing rate~$\delta_n$.

\textit{Regular boundaries in the support of Dirichet measures.}
Now, to control $A_n(G,G')$, we need the second
piece of the argument, which establishes the existence
of a robust test that can be used to distinguish
a Dirichlet measure $\DP_{\alpha G}$ from a class of Dirichet\
measures $\Cclass= \{\DP_{\alpha' G'} | G' \in\Pcal(\Theta)\}$,
where the robustness here is measured by Wasserstein
metric $\Wr$ on $\Pcal(\Theta)$. The robustness is needed
to account for the possible loss of accuracy $\delta_n$ incurred
by demixing, as alluded to in the previous paragraph.
A formal theory of robust tests is developed
in Section~\ref{Sec-boundary}. Central to this theory is
a notion of regularity
for a given class of Dirichlet measures $\Cclass$
with respect to a fixed Dirichlet measure $\Dcal:= \DP_{\alpha G}$.
In particular, we say that $\Cclass$ has regular boundary
with respect to $\Dcal$ if for each element $\Dcal' = \DP_{\alpha'G'}
\in\Cclass$
there is a measurable subset $\BB\subset\Pcal(\Theta)$
for which\vadjust{\goodbreak} the following holds: (i) $\Dcal'(\BB) - \Dcal(\BB)
\gtrsim\Wr^r(G,G')$ and (ii)\vspace*{2pt}
\[
\Dcal(\BB_\delta\setminus\BB) \rightarrow0\vspace*{2pt}
\]
as $\delta\rightarrow0$. Set $\BB$ can be thought of as a
test set which is used to approximate the variation distance
between a fixed $\Dcal$ and an arbitrary $\Dcal'$ which varies in
$\Cclass$.
$\BB_\delta$ is defined to be the set of all $P\in\Pcal(\Theta)$
for which there is a $Q\in\BB$ and $\Wr(Q,P)\leq\delta$.
Various forms of regularity are developed, which specifies how
fast the quantity in the previous display tends to 0. Thus, the
achievement of this section is to show that the regularity behavior is
closely tied to the geometry of the support of base measure $G$.
Theorems~\ref{Lem-regularity} and~\ref{Thm-sparse}
provide a complete picture of regularity for the case
$G$ has finite support, and the case $G$ has infinite and
geometrically sparse support. Now,
by controlling the rate at which $\Dcal(\BB_\delta\setminus\BB)$
tends to 0, we can control the rate at which $A_n(G,G')$
tends to 0, completing the proof of \eqref{Eqn-overview-1}.

\textit{Posterior concentration proofs.}
With the tools and inequalities established in
Section~\ref{Sec-transport} at our disposal, the proof of Theorem~\ref
{Thm-main-0}
is easily available by appealing to a general
theorem for establishing posterior concentration of a
density \cite{Ghosal-Ghosh-vanderVaart-00}, and verifying the
sufficient conditions in terms of entropy numbers, the prior thickness in
Kullback--Leibler divergence, and so on.
The proof of Theorem~\ref{Thm-main-1} follows by combining the
result from Theorem~\ref{Thm-main-0} with Theorem~\ref{Thm-contraction}
described above.

Finally, the proof of Theorem~\ref{Thm-main-2} follows from a posterior
concentration result for the mixing measure $Q$, which
is distributed by the prior $\DP_{\alpha G}$, conditionally given
the event that the base measure $G$ is perturbed by a small Wasserstein
distance $W_1$ from $G_0$ that has $k < \infty$ support points;
see Lemma~\ref{Thm-perturb} in Section~\ref{Sec-perturb}.
The proof of this lemma also follows the standard strategy
of the posterior concentration proof mentioned earlier. The main
novelty lies in the construction of a sieves of subsets of
$\Pcal(\Theta)$ which yields favorable rates of posterior
concentration. This construction is possible by showing
that the Dirichlet measure
places most its mass on subsets (of $\Pcal(\Theta)$) which can be
covered by a relatively small number of balls in $\Wr$.
Such results about the Wasserstein geometry of
the support of a Dirichlet measure may be of independent interest,
and are collected in Section~\ref{sec-geometry}.

Due to the large number of technical results, many of which
are new and rather nonstandard,
for the ease of the readers we
include the following chart that illustrates the dependence
structures of the main theorems and accompanying lemmas.
Also included are several existing theorems (in bold) upon which
our results are built in crucial ways.

\noindent\[
\fontsize{8}{10}\selectfont
\xymatrix@C=0.60cm @R=0.60cm{
Lemma~\ref{Lem-dp-hierarchy}\ar[r]\ar[dddr]\ar[dddddrrr]\ar[dddrrr] &
Lemma~\ref{Lem-W-ub}\ar[r]\ar[ddddddrr]\ar[drr] &
Lemma~\ref {Lem-entropy} \ar[dr]\ar[drr]&
{\mbox{\cite{Ghosal-Ghosh-vanderVaart-00}}
 (Theorem~2.1)}\ar@{=>}[dr] & \\
& & Lemma~\ref{Lem-DP-lowerbound}\ar[r] & Lemma~\ref{Thm-KL}\ar[r] & Theorem~\ref{Thm-main-0}\ar[ddd]\\
\mbox{\textbf{\cite{Nguyen-hdpsupp}\  Lemma~2.1}}\ar[r]\ar[dr] & Lemma~\ref{Lem-samespt}\ar[d]\ar[drr] & & &\\
\mbox{\textbf{\cite{Wong-Shen-95}\ (Theorem~2)}}\ar@{=>}[dr]& Theorem~\ref{Lem-regularity}\ar[drr] \ar[rr]& & Theorem~\ref{Thm-sparse}\ar[d] & \\
\mbox{\textbf{\cite{Nguyen-11}\ (Theorem~2)}}\ar@{=>}[r] & Lemma~\ref{Lem-wong-shen}\ar[rr] & & Theorem~\ref{Thm-contraction}\ar[r] &Theorem~\ref{Thm-main-1}\ar[dd] \\
& Lemma~\ref{Lem-DP-concent}\ar[drr] & Lemma~\ref{Lem-DP-Approx}\ar[r] & Lemma~\ref{Lem-DP-Approx-2}\ar[d] & \\
\mbox{\textbf{\cite{Nguyen-11}\ (Theorem~4)}}\ar@{=>}[rrr] & & & Lemma~\ref{Thm-perturb}\ar[r] & Theorem~\ref{Thm-main-2}
}
\]
%

\subsection{Concluding remarks and further development}

In this paper, we study posterior concentration behaviors for the base measure
of a Dirichlet measure and related quantities,
given observations associated with
sampled Dirichlet processes, using
tools developed with optimal transport distances. There are a number
of open questions that remain.
First, regarding Theorem~\ref{Thm-main-0}, we still do not know
whether the established (upper bound) of the concentration rate is optimal
or not, with respect to the number $m$ of groups,
and more interestingly with respect to the sample size $n$ per group.
Perhaps a proper way to address this question is
to directly develop a minimax optimal theory for the variables
residing in latent hierarchies of models such as the one we have considered.
Second, regarding Theorem~\ref{Thm-main-1}, our result is
applicable only in the setting where both $m$ and $n$ grow,
not the case where $m$ grows and $n$ is fixed.
Our proof method is not capable of saying much on the latter setting.
Finally, it may be of interest to consider
the problem of estimating a nonatomic base measure,
while the Dirichlet processes are not directly observed.

\section{Transportation distances of Bayesian hierarchies}
\label{Sec-transport}
Let $\Theta$ be a complete separable metric space (i.e., $\Theta$ is a
Polish space)
and $\Pcal(\Theta)$ be the space of Borel probability measures on
$\Theta$.
The weak topology on $\Pcal(\Theta)$ (or narrow topology) is induced by
convergence
against $\Cb(\Theta)$, that is, bounded continuous test functions on
$\Theta$.
Since $\Theta$ is Polish, $\Pcal(\Theta)$ is itself a Polish space.
$\Pcal(\Theta)$ is metrized by the $W_r$ Wasserstein distance: for
$G,G' \in\Pcal(\Theta)$ and $r \geq1$,
\[
\Wr \bigl(G,G' \bigr) = \inf_{\coupling\in\TPlan(G,G')} \biggl[\int
\bigl\| \theta- \theta'\bigr\|^r \,\mathrm{d}\coupling \bigl(\theta,
\theta' \bigr) \biggr]^{1/r}.
\]

By a recursion of notation, $\Pcal(\Pcal(\Theta))$ is defined as
the space of Borel probability measures on $\Pcal(\Theta)$. This is
a Polish space, and will be endowed again with a Wasserstein metric that
is induced by metric $\Wr$ on $\Pcal(\Theta)$:
%
%
\begin{equation}
\label{Eqn-W2-def} \Wcalr \bigl(\Dcal,\Dcal' \bigr) = \inf
_{\Kcal\in\TPlan(\Dcal,\Dcal')} \biggl[ \int\Wr^r \bigl(G,G'
\bigr) \,\mathrm{d}\Kcal \bigl(G,G' \bigr) \biggr]^{1/r}.
\end{equation}

We can safely reuse notation $\Wr$ as the context is clear from
the arguments.
Since the cost function $\|\theta-\theta'\|$ is continuous,
the existence of an optimal coupling $\kappa\in\TPlan(G,G')$ which
achieves the infimum
is guaranteed due to the tightness of $\TPlan(G,G')$
(cf. Theorem~4.1 of \cite{Villani-08}).
Moreover, $\Wr(G,G')$ is a continuous function and $\TPlan(\Dcal
,\Dcal')$
is again tight, so the existence of an optimal coupling in $\TPlan
(\Dcal
,\Dcal')$
is also guaranteed.

Now we present a lemma on a monotonic property of
Wasserstein metrics defined along the recursive construction
for every pair of centered random measures on $\Theta$. Part (b)
highlights a very special property of the Dirichlet measure.
In what follows, $P$ denotes a generic measure-valued random variable.
By $\int P \,\mathrm{d}\Dcal= G$ we mean $\int P(A) \,\mathrm
{d}\Dcal= G(A)$
for any measurable subset $A \subset\Theta$.

%
\begin{lemma}
\label{Lem-dp-hierarchy}
%
%
\textup{(a)} Let $\Dcal,\Dcal' \in\Pcal(\Pcal(\Theta))$ such that
$\int
P \,\mathrm{d}\Dcal= G$
and $\int P \,\mathrm{d}\Dcal' = G'$. For $r \geq1$, if $\Wr(\Dcal
,\Dcal')$
is finite
then $\Wcalr(\Dcal,\Dcal') \geq\Wr(G,G')$.

\textup{(b)} Let $\Dcal= \DP_{\alpha G}$ and $\Dcal' = \DP_{\alpha G'}$.
Then $\Wcalr(\Dcal,\Dcal') = \Wr(G,G')$ if both quantities are finite.
%
\end{lemma}

%

Recall the generative process defined by equations \eqref{Eqn-HM-A}
and \eqref{Eqn-HM-B}: The marginal density $
{p_{Y_{[n]}|G}}$ is obtained
by integrating out random measures $Q$, which is distributed
by $\DP_{\alpha G}$; see equation~\eqref{Eqn-pG}.
By a repeated application of Jensen's inequality,
it is simple to establish upper bounds on Kullback--Leibler distance
$K({p_{Y_{[n]}|G}},{p_{Y_{[n]}|G'}})$ and other
related distances in terms of transportation distance
between $G$ and $G'$.

%
\begin{lemma}
\label{Lem-W-ub}
\textup{(a)} Under assumption \textup{(A1)},
\begin{eqnarray*}
K({p_{Y_{[n]}|G}},{p_{Y_{[n]}|G'}}) & \leq&
C_1 n \Wr^r \bigl(G,G' \bigr),
\\
h^2({p_{Y_{[n]}|G}},{p_{Y_{[n]}|G'}}) &
\leq& C_1 n W_{2r}^{2r} \bigl(G,G'
\bigr),
\\
h^2({p_{Y_{[n]}|G}},{p_{Y_{[n]}|G'}}) &
\leq& V({p_{Y_{[n]}|G}},{p_{Y_{[n]}|G'}}) \leq \sqrt
{1 - \bigl(1- C_1 W_{2r}^{2r}
\bigl(G,G' \bigr) \bigr)^n}.
\end{eqnarray*}
\textup{(b)} Under assumption \textup{(A2)}, we have $\chi(
{p_{Y_{[n]}|G}},{p_{Y_{[n]}|G'}}) \leq M^n$.
\end{lemma}

The following lemma establishes
an estimate of the entropy number for the space of marginal
densities $\{{p_{Y_{[n]}|G}}| G\in\Pcal(\Theta)\}$. Part
(a) gives a very
general entropy bound. Tightened bounds are possible given
when more is known either about the space of $G$, or the
kernel density $f$. These entropy bounds have direct
consequences on the kind of concentration rates that we will get
in Theorem~\ref{Thm-main-0}.

%
\begin{lemma}
\label{Lem-entropy}
\textup{(a)} Under assumption \textup{(A1)}, for any $\varepsilon\in(0,1/2)$,
\[
\log N \bigl(\varepsilon, \bigl\{{p_{Y_{[n]}|G}}| G \in \Pcal(
\Theta) \bigr\}, h \bigr) \leq \bigl(2C_1 n \operatorname{diam}(
\Theta)/\varepsilon^2 \bigr)^{d} \log \bigl( \mathrm{e}+
2\mathrm{e}C_1 n \operatorname{diam}(\Theta)/\varepsilon^2 \bigr).
\]
\textup{(b)} Under assumption \textup{(A1)}, for any $\varepsilon\in(0,1/2)$, $k \in
\Nat$,
\begin{eqnarray*}
&&\log N \bigl(\varepsilon, \{{p_{Y_{[n]}|G}}| G \mbox { has } k
\mbox{ support points on } \Theta\}, h \bigr)
\\
&&\quad\leq kd\log \bigl(2C_1 n \operatorname{diam}(\Theta)/
\varepsilon^2 \bigr) + \log \bigl(\mathrm{e}+ 2\mathrm{e}C_1 n
\operatorname{diam}(\Theta)/ \varepsilon^2 \bigr).
\end{eqnarray*}
\textup{(c)} If $f$ is a Gaussian kernel on $\real^d$,
$f(x) = \frac{1}{(2\pi)^{d/2}\sigma^{d}}
\mathrm{e}^{-\|x\|^2/2\sigma^2}$, for some $\sigma>0$, then
\[
\log N \bigl(\varepsilon, \bigl\{{p_{Y_{[n]}|G}}| G \in \Pcal(
\Theta) \bigr\}, h \bigr) \lesssim \bigl(\log(1/\varepsilon) \bigr)^{2d+1}n^{2d}
\log n,
\]
where the multiplying constant depends only on $d, \sigma, \Theta$
(and not on $n$).
\end{lemma}


Next, define the Kullback--Leibler neighborhood of a given $G_0 \in
\Pcal(\Theta)$
with respect to $n$-vector $\YYn$ as follows:
%
%
\begin{equation}
\label{Eqn-def-KL} B_K(G_0,\delta) = \bigl\{G \in\Pcal(
\Theta)| K(p_{\YYn|G_0},p_{\YYn|G}) \leq\delta^2,
K_2(p_{\YYn|G_0},p_{\YYn|G}) \leq\delta^2
\bigr\}.
\end{equation}
%
%
The following result gives probability bound on small balls as defined by
Wasserstein metric (Lemma~5 of \cite{Nguyen-11}):

%
\begin{lemma}
\label{Lem-DP-lowerbound}
Suppose that $\law(G) = \DP_{\gamma H}$, where $H$ is a nonatomic probability
measure on $\Theta$. For a small $\varepsilon> 0$, let $D =
D(\varepsilon
,\Theta,\|\cdot\|)$
the packing number of $\Theta$ under $\|\cdot\|$. Then, for any $\Gz
\in\Pcal(\Theta)$,
\[
\Prob \bigl(G\dvt\Wr^r(G_0,G) \leq \bigl(2^r
+ 1 \bigr)\varepsilon^r \bigr) \geq\frac{\Gamma(\gamma) \gamma
^{D}}{(2D)^{D-1}} \biggl(
\frac{\varepsilon}{\operatorname{diam}(\Theta)} \biggr)^{r(D-1)} \sup_{S} \prod
_{i=1}^{D} H(S_i).
\]
Here, $(S_1,\ldots, S_D)$ denotes the $D$ disjoint $\varepsilon/2$-balls
that form
a maximal packing of $\Theta$.
$\Gamma$ denotes the gamma function. The
supremum is taken over all packings $S:= (S_1,\ldots, S_D)$.
\end{lemma}

Combine the previous lemmas to obtain an estimate of the thickness
of the hierarchical Dirichlet prior:

%
\begin{lemma}
\label{Thm-KL}
Given assumptions \textup{(A1)--(A3)}, $\Theta$ a bounded subset of $\real^d$.
\begin{longlist}[(a)]
\item[(a)] Let $D := (\Diam(\Theta))^d (n^3/\delta^2)^{d/r}$
and constants $c, C$ depending only on $C_1, M, \eta_0, \gamma,\break
\operatorname{diam}(\Theta)$ and $r$. Then,
for any $G_0 \in\Pcal(\Theta)$,
$\delta> 0$ and $n > C\log(1/\delta)$, the following
inequality holds under the probability measure $\Dcal_{\gamma H}$:
\[
\log\Prob \bigl(G \in\Bky(\Gz,\delta) \bigr) \geq c \log \bigl[
\gamma^D \bigl(\delta^2/n^3
\bigr)^{(1+d/r)(D-1) + Dd/r} \bigr].
\]
\item[(b)] If in addition, $G_0$ has exactly $k$ support points in
$\Theta$,
then
\[
\log\Prob \bigl(G \in\Bky(\Gz,\delta) \bigr) \geq c \log \bigl[
\gamma^k \bigl(\delta^2/n^3
\bigr)^{kd/r + k/r} \bigl(1/k\operatorname{diam}\Theta^r
\bigr)^{k} \bigr].
\]
\item[(c)] If $f$ is the Gaussian kernel (given in Lemma~\ref{Lem-entropy}),
then for any $G_0 \in\Pcal(\Theta)$, the bound in part \textup{(b)} of the lemma
continues to hold with $k \lesssim(\log(1/\delta))^{2d}(nd)^{2d}$.
\end{longlist}
\end{lemma}

The proofs of all lemmas presented in this section are
deferred to 
\cite{Nguyen-hdpsupp}.

\begin{pf*}{Proof of Theorem \protect\ref{Thm-main-0}}
The proof is a straightforward application of a standard result in
Bayesian asymptotics for density estimation. In particular, we shall
appeal to Theorem~2.1 of~\cite{Ghosal-Ghosh-vanderVaart-00}.
First, let $n$ be fixed, so that $n$ acts as the (fixed) dimensionality
of the $n$-vector $\YYn$. According to this theorem, as sample size
$m$ tends to infinity, as long as the constructed rate sequence $\epsmn$
satisfies the entropy condition on the class of marginal densities:
\[
\log D \bigl(\epsmn, \bigl\{\PG| G \in\Pcal(\Theta) \bigr\}, h \bigr) \leq m
\epsmn^2
\]
and the condition on the prior thickness:
\[
-\log\Prob \bigl(G \in B_K(G_0,\epsmn) \bigr) \leq M m
\epsmn^2
\]
for some universal constant $M > 0$, then the conclusion of
Theorem~\ref
{Thm-main-0}
is established for some sufficiently large constant $C>0$ not
depending on $m$ or $n$. Indeed, the entropy condition is an immediate
consequence of Lemma~\ref{Lem-entropy}, while the prior thickness condition
is immediate from Lemma~\ref{Thm-KL}.
Finally, an examination of the proof of \cite{Ghosal-Ghosh-vanderVaart-00}
reveals that the conclusion also holds by allowing $n$ to vary
as a function of $m$.
\end{pf*}

\section{Regular boundaries in the support of Dirichlet measures}
\label{Sec-boundary}

In this section, we study the property of the boundary of certain sets
(of measures)
which can be used to test one Dirichlet measure against
another. Typically, such a test set
can be defined via the variational distance between
the two measures.
However, for the purpose of subsequent
development we need a more robust test in which the robustness
can be expressed in terms of the measure of the
test set's perturbation along its boundary.
Recall the variational distance between
$\Dcal, \Dcal' \in\Pcal(\Pcal(\Theta))$ is given by
\[
V \bigl(\Dcal, \Dcal' \bigr) = \sup_{\BB\subset\Pcal(\Theta)} \bigl|
\Dcal(\BB) - \Dcal'(\BB)\bigr|.
\]
Here, the supremum is taken over all Borel measurable sets
$\BB\subset\Pcal(\Theta)$.
In what follows, fix $r\geq1$. For a subset $\BB\subset\Pcal(\Theta)$
the boundary set $\bd\BB$ is defined as the set of all elements $P
\in
\Pcal(\Theta)$
such that every $\Wr$ neighborhood for $P$ has nonempty intersection
with $\BB$ as well as the complement set $\BB^c = \Pcal(\Theta
)\setminus\BB$.

The primary objects in consideration are a pair of $(\Dcal,\Cclass)$,
with $\Dcal\in\Pcal(\Theta)$, $\Cclass\subset\Pcal(\Pcal(\Theta
))$, where
$\Dcal= \DP_{\alpha G}$ for some fixed $G \in\Pcal(\Theta)$ and
$\alpha> 0$.
$\Cclass$ is a class of Dirichlet measures
$\Cclass:= \{\DP_{\alpha' G'}| G' \in\Gcal, \alpha' > 0\}$
for some fixed $\Gcal\subset\Pcal(\Theta)$.

%
\begin{definition}
\label{Def-margin}
A class $\Cclass\subset\Pcal(\Pcal(\Theta))$ of Dirichlet measures
is said to have ${\alpha^*}$-regular boundary with respect to $\Dcal
= \DP_{\alpha G}$ for some constant ${\alpha^*}> 0$, if
there are positive constants $C_0, c_0$ and $c_1$ dependent only on
$\Dcal$ such that
for each $\Dcal' = \DP_{\alpha' G'} \in\Cclass$
there exists a measurable subset
$\BB\subset\Pcal(\Theta)$ for which the following hold:
\begin{longlist}[(iii)]
\item[(i)] $\Dcal'(\BB) - \Dcal(\BB) \geq c_0 \Wr^r(G,G')$,
\item[(ii)] $\Dcal(\BB_\delta\setminus\BB) \leq C_0
(\delta
/\Wr(G,G')
)^{\alpha^*}$ for any $\delta\leq c_1 \Wr(G,G')$.
\end{longlist}
$\Cclass$ is said to have \emph{strong}
${\alpha^*}$-regularity with respect to $\Dcal$ if condition \textup{(ii)} is
replaced by
\begin{longlist}[(iii)]
\item[(iii)] $\Dcal(\BB_\delta\setminus\BB) \leq C_0 \delta
^{\alpha^*}$ for
any $\delta\leq c_1$.
\end{longlist}
$\Cclass$ is said to have \emph{weak} regularity with respect to
$\Dcal
$ if
condition \textup{(ii)} is replaced by
\begin{longlist}[(iv)]
\item[(iv)] $\Dcal(\BB_\delta\setminus\BB) = \mathrm{o}(1)$ as
$\delta
\rightarrow0$.
\end{longlist}
\end{definition}

\begin{remark*} The nontrivial requirement here is that
constants $C_0, c_0$ and $c_1$ are independent of $\Dcal' \in\Cclass$.
Consider the following example: $\Gcal:= \{G' \in\Pcal(\Theta)|
\supp G' \cap\supp G = \varnothing\}$. Take $\Dcal' := \DP_{\alpha
' G'}$
for some $G' \in\Gcal$.
By a standard fact of Dirichlet measures
(e.g., see Theorem~3.2.4 of~\cite{Ghosh-Ramamoorthi-02}),
$\supp\Dcal= \{P\dvt\supp P \subset\supp G\}$ and
$\supp\Dcal' = \{P\dvt\supp P \subset\supp G'\}$. Thus, we
also have $\supp\Dcal\cap\supp\Dcal' = \varnothing$.
It follows that $V(\Dcal, \Dcal') = 1$. If we choose
$\delta_1 = \inf_{\theta\in\supp G; \theta' \in\supp G'} \|\theta-
\theta'\|
> 0$, and let $\BB= (\supp\Dcal')_{\delta_1/2}$,
then $\Dcal'(\BB) = 1$ and $\Dcal(\BB) = 0$. Moreover,
for any $\delta\leq\delta_1/4$, $\Dcal(\BB_\delta) = 0$, so
$\Dcal(\BB_\delta\setminus\BB) = 0$. At the first glance,
this construction appears to suggest that
$\Cclass:= \{\DP_{\alpha' G'}| G' \in\Gcal\}$
has (strong) ${\alpha^*}$-regular boundary with $\Dcal$
for any ${\alpha^*}> 0$. This is not the case,
because it is not possible to guarantee that $\delta_1 > c_1 \Wr(G,G')$
for some $c_1$ independent of $G'$. That is,
$\delta_1$ can be arbitrarily close to 0
even as $\Wr(G,G')$ remains bounded away from 0.
\end{remark*}

\subsection{The case of finite support}

We study the regularity of boundaries for the pair $(\Dcal,\Cclass)$, where
the base measure $G$ of $\Dcal= \DP_{\alpha G}$ has a finite number of
support points, while class $\Cclass$ consists of Dirichlet measures
$\Dcal' = \DP_{\alpha G'}$ where $G'$ may have infinite support
in $\Theta$. In the following subsection,
we extend the theory to handle the
case that $G$ has infinite and geometrically sparse support.

%
\begin{theorem}
\label{Lem-regularity}
Suppose that $\Theta$ is bounded.
Let $\Dcal= \DP_{\alpha G}$, where $G= \sum_{i=1}^{k}\beta_i \delta
_{\theta_i}$ for
some $k < \infty$ and $\alpha\in(0,1]$. Let $\alpha_1 > \alpha_0 > 0$
be given. Define
\[
\Cclass= \bigl\{ \DP_{\alpha' G'} | G' \in\Pcal(\Theta);
\alpha' \in[\alpha_0,\alpha_1] \bigr\}.
\]
Then $\Cclass$ has $\alpha^* r$-regular boundary with respect to
$\Dcal$,
where $\alpha^* = \min_{i} \alpha\beta_i$.
\end{theorem}

\begin{pf}
Take any $G' \in\Pcal(\Theta)$. Let $\varepsilon:= \Wr(G,G')$.
Choose constants
$c_1, c_2$ such that $c_1^r + c_2 \Diam(\Theta)^r \leq1/2^{r}$
and $c_1 \Diam(\Theta) < m:= \min_{1 \leq i\neq j \leq k} \|\theta
_i -
\theta_j\|/4$.
Let $S = \bigcup_{i=1}^{k}B_i$, where $B_i$'s for $i=1,\ldots, k$
are closed Euclidean balls
of radius $c_1\varepsilon$ and centering at $\theta_1,\ldots, \theta_k$,
respectively.
Any $G' \in\Pcal(\Theta)$ admits either (A)
$G'(S^c) \geq c_2 \varepsilon^r$, or (B) $G'(S^c) < c_2 \varepsilon^r$.

\textit{Case} (A).
$G'(S^c) \geq c_2 \varepsilon^r$.
Let $\BB= \{Q \in\Pcal(\Theta)| Q(S^c) > 1/2\}$.
Clearly, $\Dcal(\BB) = 0$. Moreover, for any
$Q \in\BB$ and $Q' \in\supp\Dcal$,
$\Wr^r(Q,Q') \geq(1/2) (c_1\varepsilon)^r$. So for any $\delta<
(1/2)^{1/r}c_1 \varepsilon$,
$\Dcal(\BB_{\delta}) = 0$. Condition (ii) of Definition~\ref{Def-margin}
is satisfied.

It remains to verify condition (i).
If $G'(S) = 0$, then $G'(S^c) = 1$ and $\Dcal'(\BB) = 1$.
So, $\Dcal'(\BB)-\Dcal(\BB) = 1$.
On the other hand, if $G'(S) > 0$ and suppose that
$\law(Q) = \Dcal'$, then $\law(Q(S)) =
\operatorname{Beta}(\alpha' G'(S), \alpha' G'(S^c))$. So,
\begin{eqnarray*}
\Dcal'(\BB) & = & \int_{0}^{1/2}
\frac{\Gamma(\alpha' )}{\Gamma
(\alpha' G'(S))
\Gamma(\alpha' G'(S^c))} x^{\alpha' G'(S) - 1}(1-x)^{\alpha'
G'(S^c)-1} \,\mathrm{d}x
\\
&\geq& \frac{(1/2)^{\alpha'}\Gamma(\alpha' )}{\Gamma(\alpha'
G'(S))\Gamma
(\alpha' G'(S^c))} \int_{0}^{1/2}
x^{\alpha' G'(S)-1} \,\mathrm{d}x
\\
& = & \frac{(1/2)^{\alpha'}\Gamma(\alpha' )}{\Gamma(\alpha'
G'(S))\Gamma
(\alpha' G'(S^c))} \times\frac{(1/2)^{\alpha' G'(S)}}{\alpha' G'(S)}
\\
& = & \frac{(1/2)^{\alpha'+\alpha'G'(S)}\Gamma(\alpha' ) \alpha
'G'(S^c)}{\Gamma(\alpha' G'(S)+1)\Gamma(\alpha' G'(S^c)+1)}
\\
& \geq& \frac{(1/2)^{2\alpha'}\Gamma(\alpha' )\alpha' G'(S^c)}{
\max_{1\leq x \leq\alpha'+1} \Gamma(x)^2} \geq\frac
{(1/2)^{2\alpha'}\Gamma(\alpha' )\alpha' c_2\varepsilon^r}{
\max_{1\leq x \leq\alpha'+1} \Gamma(x)^2}.
\end{eqnarray*}
In the above display, the first inequality is due to $(1-x)^{\gamma}
\geq1$
if $\gamma\leq0$, and $(1-x)^{\gamma} \geq(1/2)^\gamma$ if $\gamma>
0$ for
$x \in[0,1/2]$.
The third equality is due to $x\Gamma(x) = \Gamma(x+1)$ for any $x > 0$.
Condition~(i) is verified.

\textit{Case} (B). $\beta'_0 := G'(S^c) < c_2 \varepsilon^r$.
Let $\beta'_i = G'(B_i)$ for $i=1,\ldots, k$.
Consider the map
$\Phi\dvtx\Pcal(\Theta) \rightarrow\Delta^{k-1}$, defined by
\[
\Phi(Q):= \bigl(Q(B_1)/Q(S), \ldots, Q(B_k)/Q(S) \bigr).
\]

Define $P_1 := \operatorname{Dir}(\alpha\beta_1,\ldots,\alpha\beta_k)$
and $P_2 := \operatorname{Dir}(\alpha' \beta'_1,\ldots,\alpha' \beta'_k)$.
By a standard property of Dirichlet measures, $P_1$ and $P_2$
are push-forward measures of $\Dcal$ and $\Dcal'$, respectively, by
$\Phi$.
(i.e.,
if $\law(Q) = \Dcal$, then $\law(\Phi(Q)) = P_1$.
If $\law(Q) = \Dcal'$ then $\law(\Phi(Q)) = P_2$.)
Define
\[
B_1 := \biggl\{\vecq\in\Delta^{k-1} \bigg| \frac{\mathrm{d} P_2}{\mathrm{d} P_1}(
\vecq) > 1 \biggr\}.
\]
%
(This is exactly the same set defined by equation (4) of \cite{Nguyen-hdpsupp} in
the proof of
Lemma~\ref{Lem-samespt} that we shall encounter in the sequel.)
Now let $\BB= \Phi^{-1}(B_1)$. Then we have $\Dcal'(\BB) - \Dcal
(\BB)
= P_2(B_1) - P_1(B_1)
= V(P_1,P_2)$.

To verify condition (ii) of Definition~\ref{Def-margin}, recall that
\[
\Dcal(\BB_\delta\setminus\BB) = \Dcal \Biggl( \Biggl\{Q = \sum
_{i=1}^{k}q_i \delta_{\theta_i} \bigg| Q
\notin\BB; \Wr \bigl(Q,Q' \bigr) \leq\delta\mbox{ for some }
Q' \in\BB \Biggr\} \Biggr).
\]
For a measure of the form $Q = \sum_{i=1}^{k}q_i \delta_{\theta_i}$,
$\Wr(Q,Q') \leq\delta$ entails $Q(B_i) - Q'(B_i) = q_i - Q'(B_i)
\leq
\delta^r/(c_1\varepsilon)^r$, and $Q'(B_i) - q_i \leq\delta
^r/(m-c_1\varepsilon)^r
< \delta^r/(c_1\varepsilon)^r$,
for any $i=1,\ldots, k$. As well, $Q'(S^c) \leq
\delta^r/(c_1\varepsilon)^r$. This implies that
\[
\bigl|Q(B_i)/Q(S) - Q'(B_i)/Q'(S)\bigr|
= \biggl|q_i - \frac{Q'(B_i)}{1-Q'(S^c)} \biggr| \leq\frac{2\delta
^r/(c_1\varepsilon)^r}{1-\delta^r/(c_1\varepsilon)^r} \leq4
\delta^r/(c_1\varepsilon)^r,
\]
where the last inequality holds as soon as $\delta\leq c_1\varepsilon
/2^{1/r}$.
In short, $\Wr(Q,Q') \leq\delta$ implies that $\|\Phi(Q) - \Phi
(Q')\|
_\infty\leq
4\delta^r/(c_1\varepsilon)^r$. We have
\begin{eqnarray*}
\Dcal(\BB_\delta\setminus\BB) & \leq& \Dcal \bigl( \bigl\{ Q | Q\notin
\BB; \bigl\|\Phi(Q)-\Phi \bigl(Q' \bigr)\bigr\|_\infty\leq4
\delta^r/(c_1\varepsilon)^r \mbox{ for some }
Q' \in\BB \bigr\} \bigr)
\\
& = & P_1 \bigl( \bigl\{\vecq| \vecq\notin B_1; \bigl\|\vecq-
\vecq'\bigr\|_\infty\leq4\delta^r/(c_1
\varepsilon)^r \mbox{ for some } \vecq' \in
B_1 \bigr\} \bigr)
\\
& \leq& C_0 (\delta/\varepsilon)^{\alpha^*r}.
\end{eqnarray*}
The equality in the previous display is due to the definition of $\BB$,
while the last inequality is essentially the proof of Lemma~\ref
{Lem-samespt}(b).
$C_0$ is a positive constant dependent only on $\Dcal$.


It remains to verify
condition (i) in Definition~\ref{Def-margin}. We have
%
%
\begin{eqnarray}\label{Eqn-vbound}
V(P_1,P_2) & = & V(\DP_{\sum_{i=1}^{k}\alpha\beta_i \delta_{\theta_i}},
\DP_{\sum_{i=1}^{k}\alpha' \beta'_i \delta_{\theta_i}})
\nonumber
\\
& \geq& \frac{1}{(2\operatorname{diam}(\Theta))^r} \Wr^r(\DP _{\alpha G},
\DP_{\sum_{i=1}^{k}\alpha' \beta'_i \delta_{\theta_i}})
\\
 & \geq&
\frac{1}{(2\operatorname{diam}(\Theta))^r} \Wr^r \Biggl(G, \sum_{i=1}^{k}
\frac{\beta'_i}{1-\beta'_0} \delta_{\theta_i} \Biggr).\nonumber
\end{eqnarray}

The first inequality in the above display is due to
Theorem~6.15 of \cite{Villani-08}, 
while the second inequality is due to Lemma~\ref{Lem-dp-hierarchy}(a).
Now, we have
\begin{eqnarray*}
\Wr^r \Biggl(G', \sum_{i=1}^{k}
\frac{\beta'_i}{1-\beta'_0} \delta_{\theta_i} \Biggr) & \leq& (c_1
\varepsilon)^r \sum_{i=1}^{k}
\biggl(\beta'_i \wedge\frac
{\beta
'_i}{1-\beta'_0} \biggr) +
\Diam( \Theta)^r \sum_{i=1}^{k}
\biggl| \beta'_i - \frac{\beta
'_i}{1-\beta'_0} \biggr|
\\
& \leq& (c_1\varepsilon)^r + \Diam(\Theta)^r
\sum_{i=1}^{k}\frac
{\beta'_i
\beta'_0}{1-\beta'_0}
\\
& \leq& \varepsilon^r \bigl(c_1^r +
c_2 \Diam(\Theta)^r \bigr) \leq\varepsilon^r/2^r.
\end{eqnarray*}
The last inequalities in the above display
is due to the hypothesis that $\beta'_0 < c_2\varepsilon^r$, and the
choice of $c_1,c_2$.
By triangle inequality,
\[
\Wr \Biggl(G, \sum_{i=1}^{k}
\frac{\beta'_i}{1-\beta'_0} \delta_{\theta_i} \Biggr) \geq\Wr \bigl(G,G'
\bigr) - \Wr \Biggl(G', \sum_{i=1}^{k}
\frac{\beta'_i}{1-\beta'_0} \delta_{\theta_i} \Biggr) \geq \varepsilon- \varepsilon/2
= \varepsilon/2.
\]
Combining with equation \eqref{Eqn-vbound},
we obtain that $\Dcal'(\BB) - \Dcal(\BB) =
V(P_1,P_2) \geq\break \frac{1}{(2\operatorname{diam}(\Theta))^r}
(\varepsilon/2)^r$.
This concludes the proof.
\end{pf}

The following lemma, which establishes strong regularity for a
restricted class of
Dirichlet measures, supplies a key argument in the proof of the previous
theorem. The proof of this lemma is quite technical and
deferred to 
\cite{Nguyen-hdpsupp}.

%
\begin{lemma}
\label{Lem-samespt}
Let $\Dcal= \DP_{\alpha G}$, where $G = \sum_{i=1}^{k} \beta_i
\delta
_{\theta_i}$
for some $k < \infty$, $\alpha, \alpha' > 0$.
Define
\[
\Cclass= \bigl\{\DP_{\alpha'G'} | G' \in\Pcal(\Theta), \supp
G' = \supp G \bigr\}.
\]
\begin{longlist}[(a)]
\item[(a)]
If $\min_{i} \alpha\beta_i \geq1$, then
$\Cclass$ has \emph{strong} $r$-regular boundary with respect to
$\Dcal$.
\item[(b)] If $\max_{i} \alpha\beta_i < 1$,
then
$\Cclass$ has \emph{strong} ${\alpha^*}r$-regular boundary with
respect to $\Dcal$,
where ${\alpha^*}= \min_{i} \alpha\beta_i$.
\end{longlist}
\end{lemma}

\subsection{The case of infinite and geometrically sparse support}

In this subsection, we study a class of base measures $G$ that
have infinite support points, but that remain amenable to our analysis of
regular boundaries. In particular,
we consider the class of sparse measures on $\Theta$
(either ordinary sparse or supersparse) given by Definition~\ref{Def-sparse}.

%
\begin{theorem}
\label{Thm-sparse}
Assume that $\Dcal=\DP_{\alpha G}$ for some $\alpha\in(0,1]$.
$\supp G$ is a $(c_1,c_2,K)$-sparse subset of a bounded space $\Theta$
and that $G$ is a sparse measure equipped with gauge function $\gauge$.
Let $\alpha_1 \geq\alpha_0 > 0$.
Then,
for any $\Dcal' \in\Cclass$, where
\[
\Cclass= \bigl\{\Dcal' = \DP_{\alpha'G'} | G'
\in\Pcal(\Theta), \alpha' \in[\alpha_0,
\alpha_1] \bigr\}
\]
there exists a measurable set $\BB\subset\Pcal(\Theta)$ for which
\begin{longlist}[(ii)]
\item[(i)] $\Dcal'(\BB) - \Dcal(\BB) \gtrsim\Wr^r(G,G')$,
\item[(ii)] for any $\delta\lesssim\Wr(G,G')$,
\[
\Dcal(\BB_\delta\setminus\BB) \lesssim24^{K(c_0\Wr(G,G'))} \times \biggl(
\frac{\delta}{\Wr(G,G')} \biggr)^{\alpha r\gauge
(c_0\Wr(G,G'))}.
\]
\end{longlist}
Here, $c_0$ and the multiplying
constants in $\lesssim$ and $\gtrsim$ depend only on $\Dcal$.
\end{theorem}
The proof of this result is similar to Theorem~\ref{Lem-regularity}
and deferred to \cite{Nguyen-hdpsupp}.

\section{Upper bounds for Wasserstein distances of base measures}
\label{Sec-contraction}
The main purpose of this section is to obtain an upper bound of
distance of
Dirichlet base measures $\Wr(G,G')$ in terms of the variational
distance of the
marginal densities of observed data $V(
{p_{Y_{[n]}|G}},{p_{Y_{[n]}|G'}})$. In particular, we will
establish an inequality of the form: for a fixed $G\in\Pcal(\Theta)$ and
\emph{any} $G' \in\Pcal(\Theta)$,
%
%
\begin{equation}
\label{Eqn-W-UB} \Wr^r \bigl(G,G' \bigr) \lesssim V(\PG,
P_{Y_{[n]}|G'}) + A_n \bigl(G,G' \bigr),
\end{equation}
where $A_n(G,G')$ is a quantity that tends to 0 as $n \rightarrow
\infty$.
The rate at which $A_n(G,G')$ tends to 0 depends on the sparse
structure of
$G$, and the smoothness of the kernel density $f(x|\theta)$.
The full details are given in the statement of Theorem~\ref{Thm-contraction}.
It is worth contrasting this to the relatively easier
inequalities in the opposite direction,
given by Lemma~\ref{Lem-W-ub}: $V(
{p_{Y_{[n]}|G}},{p_{Y_{[n]}|G'}}) \leq h(
{p_{Y_{[n]}|G}},{p_{Y_{[n]}|G'}})
\lesssim n\W_{2r}^{2r}(G,G')$ holds generally for any pair of $G,G'$.

The proof of inequality \eqref{Eqn-W-UB} hinges on the existence of
a suitable set $\BB_n \subset\Pcal(\Theta)$ measurable with respect to
(the sigma algebra induced by) the observed variables $Y_{[n]}$,
which can then be used to distinguish $G'$ from $G$, in the sense that
\[
\Wr^r \bigl(G,G' \bigr) \lesssim P_{Y_{[n]}|G'}(
\BB_n) - \PG(\BB_n) + A_n
\bigl(G,G' \bigr).
\]

In the previous section,
we have already shown the existence of subset $\BB\subset\Pcal
(\Theta
)$ for which
\[
\Wr^r \bigl(G,G' \bigr) \lesssim\Dcal'(
\BB) - \Dcal(\BB).
\]
To link up this result to the desired bound \eqref{Eqn-W-UB}, the missing
piece of the puzzle is the existence of a point estimate for the mixing measures
on the basis of observed variables $Y_{[n]}$.
In the following, we shall establish the existence of such point estimators,
which admit finite-sample probability bounds that
may also be of independent interest.

\subsection{Finite-sample probability bounds for deconvolution problem}

Let $\Qcal$ be a subset of $\Pcal(\Theta)$, and $\Fcal= \{Q*f | Q
\in
\Qcal\}$.
Let $\Qcal_k \subset\Pcal(\Theta)$ be subset of measures with at
most $k$
support points. $\Fcal_k = \{Q*f | Q \in\Qcal_k\}$.
Given an i.i.d. $n$-vector $\YYn= (Y_1,\ldots, Y_n)$ according
to the convolution mixture density $Q_0*f$ for some $Q_0 \in\Qcal$.
Let $\eta_n$ be a sequence
of positive numbers converging to zero. Following \cite{Wong-Shen-95},
we consider
an $\eta_n$-MLE (maximum likelihood estimator) $\hatf_n \in\Fcal$
such that
\[
\frac{1}{n}\sum_{i=1}^{n} \log
\hatf_n(Y_i) \geq\sup_{g \in\Fcal}
\frac{1}{n} \sum_{i=1}^{n}\log
g(Y_i) - \eta_n.
\]

By our construction, there exists $\hatQ_n \in\Qcal$ such that
$\hatf_n = \hatQ_n*f$.

%
\begin{lemma}
\label{Lem-wong-shen}
Suppose that assumption \textup{(A1)} holds for some $r\geq1, C_1 > 0$. Let $\eta
_n$ satisfy
$\eta_n \leq c_1\varepsilon_n^2$, $\varepsilon_n \rightarrow0$ at a
rate to
be specified.
Then the $\eta_n$-MLE satisfies
the following bound under $Q_0*f$-measure, for any $Q_0 \in\Qcal$:
%
%
\begin{eqnarray}
\Prob \bigl(h(\hatf_n,Q_0*f) \geq\varepsilon_n
\bigr) & \leq& 5\exp \bigl(-c_2n\varepsilon_n^2
\bigr),
\\
\label{Eqn-deconvolution} \Prob \bigl(W_2(\hatQ_n,Q_0)
\geq\delta_n \bigr) & \leq& 5\exp \bigl(-c_2n
\varepsilon_n^2 \bigr),
\end{eqnarray}
where $c_1,c_2$ are some universal positive constants. $\varepsilon_n$ and
$\delta_n$
are given as follows:
\begin{longlist}[(a)]
\item[(a)] $\varepsilon_n = C_2(\log n/n)^{r/2d}$, if $d > 2r$;
$\varepsilon_n = C_2(\log n/n)^{r/(d+2r)}$ if $d<2r$, and
$\varepsilon_n =  (\log n)^{3/4}/n^{1/4}$ if $d=2r$.
\item[(b)] $\varepsilon_n = C_2 n^{-1/2}\log n$, if $\Qcal= \Qcal_k$
and $\Fcal= \Fcal_k$ for some $k < \infty$.
\item[(c)] If $f$
is ordinary smooth with parameter $\beta> 0$,
then $\delta_n = C_3 \varepsilon_n^{{1}/{(2+\beta d')}}$ for any
$d' > d$.
\item[(d)] If $f$
is supersmooth with parameter $\beta> 0$, then
$\delta_n = C_3 [-\log\varepsilon_n]^{-1/\beta}$.
\end{longlist}
Here, $C_2,C_3$ are different constants in each case. $C_2$
depends only on $d,r,\Theta$ and $C_1$, while $C_3$ depends only on
$d,\beta, \Theta$ and $C_2$.
\end{lemma}

\begin{pf}
Recall Theorem~2 of \cite{Wong-Shen-95},
which is restated as follows:
Suppose that $\varepsilon= \varepsilon_n$ satisfies the following inequality:
%
%
\begin{equation}
\label{Eqn-entropy} \int_{\varepsilon^2/2^8}^{\sqrt{2}\varepsilon} \bigl[\log
N(u/c_3, \Fcal, h) \bigr]^{1/2} \,\mathrm{d}u \leq c_4
n^{1/2}\varepsilon^2,
\end{equation}
where $c_3$ and $c_4$ are certain universal constants (cf. Theorem~1 of
\cite{Wong-Shen-95}). Then, for some universal constants $c_1, c_2 > 0$,
if $\eta_n \leq c_1 \varepsilon_n^2$, the following probability bound
holds under $Q_0*f$-measure, for any $Q_0 \in\Qcal$,
\[
\Prob \bigl( h(\hatf_n, Q_0*f) \geq\varepsilon_n
\bigr) \leq5\exp \bigl(-c_2 n\varepsilon_n^2
\bigr).
\]

It remains to verify the entropy condition \eqref{Eqn-entropy} given
the rates specified in the statement of the present lemma. We
shall make use of the following entropy bounds (cf. Lemma~4 of \cite
{Nguyen-11}):
%
%
\begin{eqnarray}
\label{entropy-b1} \log N(2\delta, \Qcal, W_{r}) & \leq& N\bigl(\delta,
\Theta, \|\cdot\|\bigr) \log \bigl(e+e\Diam(\Theta)^{r}/
\delta^{r} \bigr),
\\
\label{entropy-b2} \log(2\delta, \Qcal_k, W_{r}) & \leq& k
\bigl(\log N\bigl(\delta, \Theta, \| \cdot\|\bigr) + \log \bigl(e + e\operatorname{diam}(
\Theta)^{r}/\delta^{r} \bigr) \bigr).
\end{eqnarray}

By assumption (A2) and Lemma~\ref{Lem-W-ub},
we have $h^2(Q*f, Q'*f) \leq C_1W_{2r}^{2r}(Q,Q')$.
This implies that
\[
N(u/c_3,\Fcal, h) \leq N \bigl( \bigl(u^2/c_3^2
C_1 \bigr)^{1/2r}, \Qcal, W_{2r} \bigr).
\]
Since $\Theta\subset\real^d$, $N(\delta, \Theta, \|\cdot\|) \leq
(\operatorname{diam}(\Theta)/\delta)^d$. So, by \eqref{entropy-b1},
\begin{eqnarray*}
&&\int_{\varepsilon^2/2^8}^{\sqrt{2}\varepsilon} \bigl[\log N \bigl(
\bigl(u^2/c_3^2C_1
\bigr)^{1/2r}, \Qcal, W_{2r} \bigr) \bigr]^{1/2} \,\mathrm{d}u
\\
&&\qquad\leq\int_{\varepsilon^2/2^8}^{\sqrt{2}\varepsilon} \biggl[N \biggl(
\frac{u^{1/r}}{2c_3^{1/r}C_1^{1/2r}}, \Theta, \|\cdot\| \biggr) \log \bigl(e+e\Diam(
\Theta)^{2r}2^{2r}c_3^2C_1/u^2
\bigr) \biggr]^{1/2}\,\mathrm{d}u
\\
&&\qquad\leq\int_{\varepsilon^2/2^8}^{\sqrt{2}\varepsilon} \bigl(2\operatorname{diam}(
\Theta) \bigr)^{d/2}c_3^{d/2r}C_1^{d/4r}u^{-d/2r}
\bigl[\log \bigl(e+e\Diam(\Theta)^{2r}2^{2r}c_3^2C_1/u^2
\bigr) \bigr]^{1/2}\,\mathrm{d}u.
\end{eqnarray*}

For equation \eqref{Eqn-entropy} to hold, it suffices to have the right-hand
side of the inequality in the above display bounded by $c_4
n^{1/2}\varepsilon^2$.
Indeed, this is straightforward to check for the rates given
in part~(a) of the lemma.

Part (b) of the lemma is proved in the same way, by invoking a tighter bound
on the covering number via equation \eqref{entropy-b2}.
Parts (c) and (d) are immediate consequences of part (a) and (b) by
invoking Theorem~2 of \cite{Nguyen-11}.
\end{pf}


\subsection{Key upper bound for the Wasserstein distance of base measures}

We are ready to prove the key theorem of this section.
%

\begin{theorem}
\label{Thm-contraction}
Suppose that $\Theta$ is a bounded subset of $\real^d$, \textup{(A1)} holds for
some $C_1 > 0$ and some $r \in[1,2]$.
Let $\delta_n$ and $\varepsilon_n$ be vanishing sequences
for which equation \eqref{Eqn-deconvolution} holds.
Fix $G\in\Pcal(\Theta)$ and $\alpha\in(0,1]$, while $G'$ varies in
$\Pcal(\Theta)$.
Let ${\alpha^*}= \alpha\inf_{\theta\in\supp G}G(\{\theta\})$.
Then there are
positive constants
$c_0,c_1,C_0$ depending only on $G$, and $c_2>0$ a universal constant, such
that for any $G' \in\Pcal(\Theta)$, $\alpha' \in[\alpha_1,\alpha
_0]$ given
and $n$ sufficiently large so that $\delta_n \lesssim\Wr(G,G')$, the
following holds:
%
%
\begin{equation}
\label{Eqn-bb} c_0\Wr^r \bigl(G,G' \bigr)
\leq V(\PG,P_{Y_{[n]}|G'}) + 10\exp \bigl(-c_2 n
\varepsilon_n^2 \bigr) + A_n \bigl(\Wr
\bigl(G,G' \bigr) \bigr),
\end{equation}
where $A_n(\Wr(G,G'))$ takes the form:
%
%
\begin{equation}
\label{Eqn-Thm-W} A_n(\omega) = \cases{ C_0(2
\delta_n/\omega)^{{\alpha^*}r}, & \quad $ \mbox{if } G \mbox{ has finite
support,}$\vspace*{2pt}
\cr
C_024^{K(c_1\omega)}(2
\delta_n/\omega)^{\alpha r\gauge(c_1\omega)}, &\quad $\mbox{if } G \mbox { is } (
\gamma_1,\gamma_2,K)\mbox{-sparse with gauge $g$}$.}
\end{equation}
\end{theorem}

\begin{pf}
Suppose that $G$ has finite support.
By Theorem~\ref{Lem-regularity} (applied for $\Wr$) there
are positive constants $C_0,c_0$ independent of $G'$ such that
for some measurable set $\BB\subset\Pcal(\Theta)$,
(i) $\Dcal'(\BB) - \Dcal(\BB) \geq c_0 \Wr^r(G,G')$ and
(ii) $\Dcal(\BB_\delta\setminus\BB) \leq C_0 (\delta/\Wr
(G,G'))^{{\alpha^*}r}$ for all
$\delta\lesssim\Wr(G,G')$.

Recall that $\hatQ_n$ is a point estimate of $Q$ defined earlier in this
section. By the definition of variational distance, for any $\delta> 0$
\[
V(\PG,P_{Y_{[n]}|G'}) \geq\Prob \bigl(\hatQ_n \in
\BB_{\delta} |G' \bigr) - \Prob(\hatQ_n \in
\BB_{\delta
} |G).
\]
Here, $\Prob(\cdot|G)$ is taken to mean the probability of an event
given that the
observations are generated according to the Dirichlet base measure $G$.
Set $\BB_\delta:= \{Q \in\Pcal(\Theta) | \mbox{ there is } Q' \in
\BB
\mbox{ such that }
\Wr(Q,Q') \leq\delta\}$. We have
\begin{eqnarray*}
\Prob \bigl(\hatQ_n \in\BB_{\delta} |G' \bigr)
& \geq& \Prob \bigl(\hatQ_n \in\BB_{\delta}, \Wr(
\hatQ_n,Q) < \delta|G' \bigr)
\\
& \geq& \Prob \bigl(Q \in\BB, \Wr(\hatQ_n,Q) < \delta|G'
\bigr)
\\
& \geq& \Dcal'(\BB) - \Prob \bigl(\Wr(\hatQ_n,Q) \geq
\delta|G' \bigr).
\end{eqnarray*}
We also have
\begin{eqnarray*}
\Prob(\hatQ_n \in\BB_{\delta} |G) & \leq& \Prob \bigl(
\hatQ_n \in\BB_{\delta}, \Wr(\hatQ_n,Q) < \delta|G
\bigr) + \Prob \bigl(\Wr(\hatQ_n,Q) \geq\delta| G \bigr)
\\
& \leq& \Prob(Q \in\BB_{2\delta}|G) + \Prob \bigl(\Wr(
\hatQ_n,Q) \geq\delta|G \bigr)
\\
& = & \Dcal(\BB_{2\delta}) + \Prob \bigl(\Wr(\hatQ_n,Q) \geq
\delta|G \bigr).
\end{eqnarray*}
Hence,
\begin{eqnarray*}
&&V(\PG,P_{Y_{[n]}|G'})\\
&&\quad  \geq \Dcal'(\BB) - \Dcal(
\BB_{2\delta}) - 2 \sup_{Q \in
\Qcal} \Prob \bigl(\Wr(
\hatQ_n,Q) \geq\delta \bigr)
\\
&&\quad \geq \bigl(\Dcal'(\BB)-\Dcal(\BB) \bigr) - \Dcal(
\BB_{2\delta}\setminus B) - 2\sup_{Q \in\Qcal} \Prob \bigl(\Wr(
\hatQ_n,Q) \geq\delta \bigr).
\end{eqnarray*}

Since $r\in[1,2]$, $\Wr(\hatQ_n,Q) \leq W_2(\hatQ_n,Q)$.
Choose $\delta:= \delta_n$ such that equation \eqref
{Eqn-deconvolution} holds.
Then, as soon as $2\delta_n \lesssim\Wr(G,G')$, for some multiplying
constant depending only on $G$, we have
\[
V(\PG,P_{Y_{[n]}|G'}) \geq c_0\Wr^r
\bigl(G,G' \bigr) - C_0 \bigl(2\delta_n/ \Wr
\bigl(G,G' \bigr) \bigr)^{{\alpha^*}r}- 10\exp
\bigl(-c_2 n\varepsilon_n^2 \bigr).
\]
The case that $G$ has infinite support
proceeds in a similar way by invoking Theorem~\ref{Thm-sparse}.
%
\end{pf}

\begin{remark*}
As we shall see shortly, Theorem~\ref{Thm-contraction}
is instrumental in the proof of Theorem~\ref{Thm-main-1}:
one can now deduce the convergence of the Dirichlet
base measure $G$ (toward $G_0$) from the convergence of the corresponding
marginal density ${p_{Y_{[n]}|G}}$ (toward $
{p_{Y_{[n]}|G_0}}$). We note that
the bound represented by \eqref{Eqn-bb} is not sharp in certain regimes,
which carry immediate consequences on
the kind of posterior concentration rates that we can obtain for $G$.
In particular, the right-hand side of inequality \eqref{Eqn-bb}
increases as $n \rightarrow\infty$, due to the fact that
$V(\PG,P_{Y_{[n]}|G'})$ typically increases as $n$ increases, while
the left-hand
side is independent of $n$.

The root of this unnatural feature is due to a simple
technique employed in the proof of Theorem~\ref{Thm-contraction},
which targets the regime that $n\rightarrow\infty$, so that one
can build on the machinery of the existence of a robust test
for Dirichlet base measures developed in Section~\ref{Sec-boundary}.
Ideally, one would like to construct
a test for base measure $G$ given $n$-vector data
$\YYn$, by integrating out the latent variable $Q$. Instead, the
bound \eqref{Eqn-bb} of Theorem \eqref{Thm-contraction} is derived by
a decoupling approach: one can first obtain a point estimate
for $Q$ on the basis of the data $\YYn$, and then relies on
the existence of a robust
test for $G$ based on the population of $Q$. Due to
the decoupling approach, we necessarily require $n$ to
grow so that the quality of the point estimate for $Q$ is
sufficiently good. An artifact of this technique, however,
is that the upper bound for $\Wr(G,G')$ can only be
derived as a summation of several quantities, two of which vanish
as $n$ increases (as desired), but the same cannot be said for the
remaining quantity, that is,
the variational distance of marginal densities of
$n$-vector $\YYn$.
\end{remark*}

\subsection{Proof of Theorem \texorpdfstring{\protect\ref{Thm-main-1}}{2.2}}

Now we are ready to prove Theorem~\ref{Thm-main-1}.
By Theorem~\ref{Thm-main-0}, as $m \rightarrow\infty$,
while $n$ either varies with $m$ or is held fixed, we have
\[
\Pi_G \bigl(V({p_{Y_{[n]}|G_0}},
{p_{Y_{[n]}|G}}) \leq\epsmn| \YYn^{[m]} \bigr)
\rightarrow1
\]
in $\PGz^m$-probability. Here, we exploit the fact that $V \leq h$.
Now, by Theorem~\ref{Thm-contraction} applied to the pair
of $G_0,G$, with the latter allowed to vary in $\Pcal(\Theta)$,
there are positive constants $c_0, c_1, C_0$ depending on $G_0$
and $c_2 > 0$ a universal constant such that
%
%
\begin{equation}
\label{Eqn-the-bound} c_0 W_1(G_0,G) \leq V(\PGz,
\PG) + 10 \exp \bigl(-c_2 n\varepsilon_n^2
\bigr) + A_n \bigl(W_1(G_0,G) \bigr),
\end{equation}
for any $G\in\Pcal(\Theta)$.
So we have
\[
\Pi_G \bigl(c_0 W_1(G_0,G)
\leq \epsmn+ 10 \exp \bigl(-c_2 n\varepsilon_n^2
\bigr) + A_n \bigl(W_1(G_0,G) \bigr) |
\YYn^{[m]} \bigr) \rightarrow1
\]
in $\PGz^m$-probability.

To derive concrete concentration rates,
consider the case $G_0$ has finite support. By Theorem~\ref{Thm-contraction}
$A_n(W_1(G_0,G)) \asymp(2\delta_n/W_1(G_0,G))^{{\alpha^*}}$.
Plugging to equation \eqref{Eqn-the-bound}, we obtain
\begin{eqnarray*}
W_1(G_0,G) &\lesssim& V(\PGz,\PG) +\exp
\bigl(-c_2n\varepsilon^2 \bigr) + \delta_n^{{\alpha^*}/({\alpha^*}+1)}
\\
&\lesssim & V(\PGz,\PG) + \delta_n^{{\alpha^*}/({\alpha^*}+1)},
\end{eqnarray*}
where we have exploited the fact that the term $\exp(-c_2n\varepsilon^2)$
is negligible compared to the remaining terms. The conclusion
of the theorem follows immediately.

Next, consider the case $G_0$ has infinite support, and in fact has
geometrically
sparse support. For the case that $G_0$ is super sparse with parameters
$(\gamma_0,\gamma_1)$, that is,
$K(\varepsilon) \lesssim[\log(1/\varepsilon)]^{\gamma_0}$,
and $\gauge(\varepsilon) \gtrsim[\log(1/\varepsilon)]^{-\gamma_1}$.
It is simple to verify that as long as $\varepsilon\gtrsim\delta_n$,
the constraint
\[
\varepsilon\lesssim A_n(\varepsilon) = 24^{K(c_1\varepsilon)} \times(2
\delta_n/\varepsilon)^{c_1 \gauge
(c_1\varepsilon)}
\]
implies that
\[
\varepsilon\lesssim\exp- \bigl[\log(1/\delta_n)
\bigr]^{1/(\gamma
_1 + 1\vee
\gamma_0)}.
\]
Thus, equation \eqref{Eqn-the-bound} entails that
\[
W_1(G_0,G) \lesssim V(\PGz,\PG) + \exp- \bigl[\log(1/
\delta_n) \bigr]^{1/(\gamma
_1 + 1\vee\gamma_0)}.
\]
For the case that $G_0$ is ordinary sparse with parameters $(\gamma
_0,\gamma_1)$,
that is $K(\varepsilon) \lesssim(1/\varepsilon)^{\gamma_0}$,
and $\gauge(\varepsilon) \gtrsim\varepsilon^{\gamma_1}$.
Similarly, note that
the inequality
\[
\varepsilon\lesssim A_n(\varepsilon)
\]
entails that
\[
\varepsilon\lesssim \bigl[\log(1/\delta_n) \bigr]^{-1/(\gamma
_1+\gamma_0)}.
\]
Thus we have shown that
\[
\Pi_G \bigl(W_1(G_0,G) \lesssim\epsmn+
\Delta_n |\YYn^{[m]} \bigr) \rightarrow1
\]
in $P^m_{\YYn|G_0}$-probability, for the choice of $\Delta_n$ given
in the
statement of the theorem.

Examples of $\varepsilon_n$ and $\delta_n$ are given in Lemma~\ref
{Lem-wong-shen}:
If $f$ is an ordinary smooth kernel density, $\log(1/\delta_n) \asymp
\frac{1}{2+\beta d'}
\log(1/\varepsilon_n) \asymp\log n$.
If $f$ is a supersmooth kernel density, $\log(1/\delta_n) \asymp
\frac
{1}{\beta}
\log\log(1/\varepsilon_n) \asymp\log\log n$.

\section{Borrowing strength in hierarchical Bayes}
\label{Sec-perturb}

This section is devoted to the proof of Theorem~\ref{Thm-main-2}.
The proof is a simple consequence from Lemma~\ref{Thm-perturb},
which establishes the posterior concentration behavior for a
mixture distribution $Q*f$, where $Q$ is a Dirichlet process
distributed by $\DP_{\alpha G}$, given that
the base measure $G$ is a small perturbation from the true
base measure $G_0$ that is now assumed to have finite support.
A~complete statement of Lemma~\ref{Thm-perturb}
is given in Section~\ref{sec-proof-perturb}.
In the following we proceed to give a proof of Theorem~\ref{Thm-main-2}.

\subsection{Proof of Theorem \texorpdfstring{\protect\ref{Thm-main-2}}{2.3}}
Recall that for each $\nn$, $\delmn= \delmn(\nn)$ is a net of scalars
indexed by $m,n$ that tend to 0.
Define $A_{mn}^{(\nn)}:= \{G\dvt W_1(G,G_0) \geq\delmn\}$ and
$B_{mn}^{(\nn)}:= \{Q_0\dvt h(Q_0*f,Q_0^**f) \geq C ((\log\nn/\nn
)^{1/(d+2)}
+ \delmn^{r/2}\log(1/\delmn))\}$
for some large constant $C$.
Due to the conditional independence of $Y_{[\nn]}^0$ and $\YYn^{[m]}$
given $G$,
\begin{eqnarray*}
 \Pi_Q \bigl(Q_0 \in B_{mn}^{(\nn)}|
Y_{[\nn]}^0, \YYn^{[m]} \bigr) &=& \int
\Pi_Q \bigl(Q_0 \in B_{mn}^{(\nn)} |
G, Y_{[\nn]}^0 \bigr) \,\mathrm{d}\Pi_G \bigl(G|
Y_{[\nn]}^0, \YYn^{[m]} \bigr)
\\
&\leq& \int_{\Pcal(\Theta)\setminus A_{mn}^{(\nn)}} \Pi_Q \bigl(Q_0
\in B_{mn}^{(\nn)} | G, Y_{[\nn]}^0 \bigr)
\,\mathrm{d}\Pi_G \bigl(G| Y_{[\nn]}^0,
\YYn^{[m]} \bigr)
\\
&&{} + \Pi_G \bigl(G \in A_{mn}^{(\nn)} |
Y_{[\nn]}^0, \YYn^{[m]} \bigr).
\end{eqnarray*}

For each $\nn$, the second quantity in the upper bound tends to 0
in $P_{Y_{[\nn]}^0|Q_0^*}\times P_{\YYn|G_0}$-probability,
as $m, n\rightarrow\infty$ at suitable rates by
condition (b) of the theorem. Now, as $\nn\rightarrow\infty$,
the first quantity tends to 0
as a consequence of Lemma~\ref{Thm-perturb}. This completes the
proof for (i). Parts (ii) and (iii) are proved in the same way.

\subsection{Wasserstein geometry of the support of a single
Dirichlet measure}
\label{sec-geometry}

Before proceeding to a proof for Lemma~\ref{Thm-perturb}, we prepare
three technical lemmas, which provide
a detailed picture of the geometry of the support of a Dirichlet
measure, and may be of independent interest.
The first lemma demonstrates gains in the thickness of the conditional
Dirichlet prior (given a perturbed base measure) compared to the
unconditional Dirichlet prior. The second and third lemma show
that Dirichlet measure concentrates most its mass on ``small'' sets,
by which we mean sets that admit a small number of covering balls
in Wasserstein metrics. This characterization
enables the construction of a suitable sieves as required by the proof
of Lemma~\ref{Thm-perturb}.

%
\begin{lemma}
\label{Lem-DP-concent}
Given $G_0 = \sum_{i=1}^{k}\beta_i \delta_{\theta_{i}}$ and small
$\varepsilon> 0$.
Let $G \in\Pcal(\Theta)$ such that $W_1(G,G_0) \leq\varepsilon$.
Suppose that $\law(Q) = \DP_{\alpha G}$, where $\alpha\in(0,1]$.
\begin{longlist}[(a)]
\item[(a)] For any $Q_0 \in\Pcal(\Theta)$ such that $\supp Q_0
\subset\supp G_0$,
and any $\delta$ such that $\delta\geq\max_{i\leq k} 2\varepsilon
/\beta
_i$ and
$\delta\leq\min_{i,j\leq k} \|\theta_i-\theta_j\|/2$, any $r\geq1$,
there holds
\[
\Prob \bigl(\Wr(Q_0,Q) \leq2^{1/r}\delta \bigr) \geq\Gamma(
\alpha) (\alpha/2)^k \biggl(\frac{\delta^r}{2k \operatorname
{diam}(\Theta)} \biggr)^{\alpha+ k-1}
\prod_{i=1}^{k}\beta_i.
\]

\item[(b)] In addition, suppose that \textup{(A1)--(A2)} hold for some $r \geq1$.
Then, there are constants $C, c>0$ depending only on $\alpha, k, C_1,
M, \operatorname{diam}(\Theta), r$
and $\beta_i$'s such that
for any $\delta$ such that $\delta/\log(1/\delta) \geq C\varepsilon^{r/2}$,
\[
\Prob \bigl(Q \in B_K(Q_0, \delta) \bigr) \geq c
\bigl( \delta/\log(1/\delta) \bigr)^{2(\alpha+k-1)}.
\]
\end{longlist}
\end{lemma}

This should be contrasted with the general small ball probability bound
of Dirichlet process as stated by Lemma~\ref{Lem-DP-lowerbound}. In that
lemma, the base measure is an arbitrary nonatomic measure, while the lower
bound is applied to any small $\Wr$ ball centering at an arbitrary measure.
The lower bound is exponentially small in the radius.
In the present lemma, the base measure $G$ is constrained
to being close to a discrete measure $G_0$ with $k <\infty$ support
points, while
the lower bound is applied to small $\Wr$ balls centering at $Q_0$
that shares
the same support as $G_0$. As a result, the lower bound is only
polynomially small
in the radius.

The following lemma relies on the intuition that
the Dirichlet measure concentrates most
its mass on probability measures which place most their mass
on a ``small'' number of support points.

%
\begin{lemma}
\label{Lem-DP-Approx}
Let $\Dcal:= \DP_{\alpha G}$ and $r\geq1$. For any $\delta> 0$,
and for any $\m\in\Nat_+$, there is a
measurable set $\Bcal_\m\subset\Pcal(\Theta)$ satisfies
the following properties:
\begin{longlist}[(a)]
\item[(a)] $\sup_{Q\in\Bcal_\m} \inf_{Q'\in\Qcal_\m} \Wr
(Q,Q') \leq
\delta$.
\item[(b)] $\log N(\delta, \Bcal_\m, \Wr)
\leq\m(\log N(\delta/4, \Theta, \|\cdot\|) + \log(e+ 4e\Diam
(\Theta
)^r/\delta^r))$.
\item[(c)] There holds
%
\[
\Dcal \bigl(\Pcal(\Theta) \setminus\Bcal_\m \bigr) \leq
\m^{-\m} \bigl(\delta/\Diam(\Theta) \bigr)^{\alpha r} \bigl[e\alpha
r \log \bigl(\operatorname{diam}(\Theta)/\delta \bigr) \bigr]^\m.
\]

\end{longlist}
\end{lemma}

To see that the set $\Bcal_\m$ has small entropy relative to $\Pcal
(\Theta)$,
we note a general estimate for $\Pcal(\Theta)$, which gives an upper bound
that is exponentially large in terms of the entropy of $\Theta$
(cf. equation \eqref{entropy-b1}):
\[
\log N \bigl(\delta,\Pcal(\Theta), \Wr \bigr) \leq N\bigl(\delta/2, \Theta, \| \cdot
\|\bigr) \log \bigl(e + 2e\Diam(\Theta)^r/\delta^r \bigr).
\]
In Lemma~\ref{Lem-DP-Approx}, the bound on entropy of $\Bcal_k$
increases only linearly in the entropy of $\Theta$.
However, it also increases with $k$, which controls the measure of
the complement of $\Bcal_k$. Next, we consider the additional assumption
that the Dirichlet base measure is a small perturbation of a discrete measure
with $k$ support points. The strength of this result
compared to the previous lemma is that the entropy estimate depends only
linearly on the entropy of $\Theta$, while $k$ is fixed. The measure
of the complement set of $\Bcal$ is controlled only
by the amount of perturbation.

%
\begin{lemma}
\label{Lem-DP-Approx-2}
Given $\varepsilon> 0$, $k < \infty, r\geq1$.
Let $G_0, G \in\Pcal(\Theta)$ such that $G_0$ has $k$ support
points and $W_1(G,G_0) \leq\varepsilon$.
Let $\Dcal:= \DP_{\alpha G}$ for some $\alpha> 0$. For any $\delta> 0$,
there is a measurable set $\Bcal\subset\Pcal(\Theta)$ that
satisfies the
following:
\begin{longlist}[(a)]
\item[(a)] $\log N(\delta, \Bcal, W_r) \leq k(\log N(\delta/4,
\Theta,
\|\cdot\|)
+ \log(e+4e\operatorname{diam}(\Theta)^r/\delta^r))$.
\item[(b)] $\Dcal(\Pcal(\Theta) \setminus\Bcal) \leq\varepsilon
\operatorname{diam}
(\Theta)^{r-1}/\delta^r$.
\end{longlist}
\end{lemma}

The proofs of all three lemmas are given in 
\cite{Nguyen-hdpsupp}.

\subsection{Posterior concentration under perturbation of base measure}
\label{sec-proof-perturb}

Here, we state a key result that is needed in the proof
of Theorem~\ref{Thm-main-2}.
%

\begin{lemma}
\label{Thm-perturb}
Let $\Theta$ be a bounded subset of $\real^d$.
Assumptions \textup{(A1)--(A2)} hold.
Let $Q_0 \in\Pcal(\Theta)$ such that $\supp Q_0 \subset\supp G_0$,
where $G_0 = \sum_{i=1}^{k}\beta_i \delta_{\theta_i}$ for some $k
<\infty$.
Let $\Pi_G$ be an arbitrary prior distribution on $\Pcal(\Theta)$.
Consider the following hierarchical model:
\begin{eqnarray*}
G \sim\Pi_G, Q|G &\sim& \Pi_Q :=
\DP_{\alpha G},
\\
\YYn= (Y_1,\ldots, Y_n)| Q & \stackrel{\mathrm{i.i.d.}} {\sim} &
Q*f.
\end{eqnarray*}
Let $\varepsilon_n \downarrow0$ and define
events $\Ecal_n:= \{W_1(G,G_0) \leq\varepsilon_n\}$.
Then the posterior distribution of $Q$ given $\YYn$
admits the following as $n\rightarrow\infty$:
%
%
\begin{eqnarray}
\label{Eqn-conv-1} \Pi_{Q} \bigl(h(Q*f,Q_0*f)\geq
\delta_{n} | \YYn, \Ecal_n \bigr) &\rightarrow&0,
\\
\label{Eqn-conv-2} \Pi_{Q} \bigl(W_2(Q,Q_0)\geq
M_n\delta_{n} | \YYn, \Ecal_n \bigr)
&\rightarrow& 0
\end{eqnarray}
in $(Q_0*f) \times\Pi_G$-probability, where the rates $\delta_n$
and $M_n\delta_n$ are given as follows:
\begin{longlist}[(iii)]
\item[(i)] $\delta_n \asymp(\log n/n)^{1/(d+2)}+ \varepsilon
_n^{r/2}\log
(1/\varepsilon_n)$.
\item[(ii)] If $f$ is ordinary smooth with smoothness $\beta> 0$,
$M_n\delta_n \asymp\delta_n^{{1}/{(2+\beta d')}}$ for any $d'>d$.
\item[(iii)] If $f$ is supersmooth with smoothness $\beta> 0$, then
$M_n \delta_n \asymp(-\log\delta_n)^{-1/\beta}$.
\end{longlist}
If $\varepsilon_n \downarrow0$ suitably fast, then the following
rates for $\delta_n$ are valid:
\begin{longlist}[(iii)]
\item[(iv)] If $f$ is ordinary smooth,
and $\varepsilon_n \rightarrow0$ sufficiently fast such that
$\varepsilon_n \lesssim n^{-(\alpha+k+4M_0)}\times\break   (\log n)^{-(\alpha
+k-2)}$, where
$M_0$ is some large constant,
then $\delta_n \asymp(\log n/n)^{1/2}$.

\item[(v)] If $f$ is supersmooth with smoothness $\beta> 0$,
and $\varepsilon_n \rightarrow0$ sufficiently fast such that
$\varepsilon_n \lesssim n^{-2(\alpha+k)/(\beta+2)} (\log
n)^{-2(\alpha+k-1)}
\exp(-4n^{\beta/(\beta+2)})$,
then $\delta_n \asymp(1/n)^{1/(\beta+2)}$.
\end{longlist}
\end{lemma}

We defer the proof of this lemma to 
\cite{Nguyen-hdpsupp}.
The basic structure contains of mostly standard calculations.
The main novel part of
the proof lies in the construction of suitable sieves that yield fast rates
of convergence. The existence of such sieves is a direct
consequence of the geometric lemmas presented in the previous
subsection.

\section*{Acknowledgements}
This research was supported in part by NSF grants CCF-1115769,
NSF CAREER DMS-1351362, and NSF CNS-1409303. The author wishes to
thank the Associate Editor and referees for many helpful comments.

\begin{supplement}
\stitle{Proofs of remaining results}
\slink[doi]{10.3150/15-BEJ703SUPP} 
\sdatatype{.pdf}
\sfilename{BEJ703\_supp.pdf}
\sdescription{Due to space constraints, we provide the proofs of the remaining
technical results of this paper in \cite{Nguyen-hdpsupp}.}
\end{supplement}


%
%



\printhistory
\end{document}